\documentclass[12 pt]{amsart}
\title{Lang-Weil Type Estimates in Finite Difference Fields}
\author{Martin Hils}
\address{Institut f\"{u}r Mathematische Logik und Grundlagenforschung, Universit\"{a}t M\"{u}nster, Einsteinstr. 62, D-48149 M\"{u}nster, Germany}
\email{hils@uni-muenster.de}
\thanks{}

\author{Ehud Hrushovski}
\address{Mathematical Institute, University of Oxford, Andrew Wiles Building,
Radcliffe Observatory Quarter, Woodstock Road, Oxford OX2 6GG, UK}
\email{ehud.hrushovski@maths.ox.ac.uk}
\author{Jinhe Ye}
\address{Mathematical Institute, University of Oxford, Andrew Wiles Building,
Radcliffe Observatory Quarter, Woodstock Road, Oxford OX2 6GG, UK}
\email{jinhe.ye@maths.ox.ac.uk}
\author{Tingxiang Zou}
\address{Mathematical Institute, University of Bonn, Endenicher Allee 60, 53115 Bonn, Germany}
\email{tzou@math.uni-bonn.de}
\usepackage{amsmath, amssymb, amsthm, enumitem, comment}
\usepackage{tikz-cd}
\usepackage[mathscr]{euscript}
\usepackage{fullpage} 	% without this, will have wide math-article-paper margins
\usepackage{hyperref}
\usepackage{lineno}
\usepackage[all]{xy}
\usepackage{centernot}
\usepackage{xcolor}

\DeclareFontFamily{U}{BOONDOX-calo}{\skewchar\font=45 }
\DeclareFontShape{U}{BOONDOX-calo}{m}{n}{
  <-> s*[1.05] BOONDOX-r-calo}{}
\DeclareFontShape{U}{BOONDOX-calo}{b}{n}{
  <-> s*[1.05] BOONDOX-b-calo}{}
\DeclareMathAlphabet{\mathcalboondox}{U}{BOONDOX-calo}{m}{n}
\SetMathAlphabet{\mathcalboondox}{bold}{U}{BOONDOX-calo}{b}{n}
\DeclareMathAlphabet{\mathbcalboondox}{U}{BOONDOX-calo}{b}{n}

\DeclareMathOperator*{\forkindep}{\raise0.2ex\hbox{\ooalign{\hidewidth$\vert$\hidewidth\cr\raise-0.9ex\hbox{$\smile$}}}}

\newcommand{\monster}{\mathfrak{C}}

\newcommand{\trfdeg}{\operatorname{trf.\!deg}}

\newcommand{\frob}{\mathrm{Frob}}
\newcommand{\inv}{\mathrm{inv}}

\newcommand{\Frac}{\operatorname{Frac}}
\newcommand{\acl}{\operatorname{acl}}

\newcommand{\tp}{\operatorname{tp}}
\newcommand{\qftp}{\operatorname{qftp}}

\newcommand{\val}{\operatorname{val}}
\newcommand{\res}{\operatorname{res}}

\newcommand{\Chara}{\operatorname{char}}

\newcommand{\ACFA}{\mathrm{ACFA}}

\newcommand{\rf}{{\mathbf k}}
\newcommand{\vf}{{\mathbf{K}}}
\newcommand{\vg}{{\mathbf \Gamma}}

\newtheorem*{claim-star}{Claim}
\newtheorem{theorem}{Theorem}[section] % numbered like the section
\newtheorem{lemma}[theorem]{Lemma}

\newtheorem{prop-def}[theorem]{Proposition-Definition}
\newtheorem{corollary}[theorem]{Corollary}
\newtheorem{fact}[theorem]{Fact}
\newtheorem{fact-eh}[theorem]{Fact(?)}

\newtheorem{question}[theorem]{Question}

\newtheorem{proposition-eh}[theorem]{Proposition(?)}
\newtheorem*{theorem-star}{Theorem}
\newtheorem*{conjecture-star}{Conjecture}
\newtheorem*{question-star}{Question}
\newtheorem*{lemma-star}{Lemma}

\newtheorem*{thmA}{Theorem A}

\newtheorem*{thmB}{Theorem B}
\newtheorem*{thmC}{Theorem C}

\newtheorem*{theorem*}{Theorem}

\theoremstyle{definition}
\newtheorem{definition}[theorem]{Definition}
\newtheorem{example}[theorem]{Example}

\newtheorem{remark}[theorem]{Remark}
\theoremstyle{remark}

\newcommand{\Aa}{\mathbb{A}}
\newcommand{\Ff}{\mathbb{F}}

\newcommand{\Qq}{\mathbb{Q}}

\newcommand{\Zz}{\mathbb{Z}}
\newcommand{\Nn}{\mathbb{N}}

\newcommand{\Oo}{\mathcal{O}}
\newcommand{\mm}{\mathfrak{m}}

\newcommand{\rld}{\mathrm{rld}}
\newcommand{\ld}{\mathrm{ld}}

\newcommand{\tdim}{\mathrm{trf.dim}}
\newcommand{\adim}{\mathrm{alg.dim}}

\begin{document}
\begin{abstract}
    We prove a uniform estimate of the number of points for difference algebraic varieties in finite difference fields in the spirit of Lang-Weil. More precisely, we give uniform lower and upper bounds for the number of rational points of a difference variety in terms of its transformal dimension.  As a main technical ingredient, we prove an equidimensionality result for Frobenius reductions of difference varieties.
\end{abstract}
\maketitle

\tableofcontents
%\pagecolor{lightgray}
\section{Introduction}

The Lang-Weil estimates give the asymptotics of the number of points of a variety in a large finite field. Here we consider a  difference variety analogue:   
the number of solutions of a number of polynomial equations involving $x^p$ 
 in a large finite field $\Ff_{p^n}$; we seek an expression uniform in $p$ and $n$. We find the right exponent, but not an exact leading coefficient; indeed no
 single limiting value of this coefficient can be expected, as arithmetic questions intervene.

The proof requires several key inputs from the model theory of difference fields, the classical Lang-Weil estimate, and some recent developments towards non-Archimedean geometry in the difference setting~\cite{DoHr22}. For an introduction to model theory, see~\cite{tent-ziegler}. For a brief introduction to the model theory of difference fields, see~\cite{zoe-icm}. Before presenting the precise results and the idea of the proof, we recall some history of the Lang-Weil estimate.
\subsection{Weil Conjectures and Lang-Weil Estimate.}
The Weil conjectures are a set of mathematical conjectures proposed by Weil \cite{weil49} in the late 1940s. These conjectures are fundamental in algebraic geometry and have profound implications for number theory. The Weil conjectures are concerned with counting the number of rational points of algebraic varieties over finite fields. Specifically, let $X$ be a geometrically integral projective algebraic variety defined over a finite field $\Ff_q$, where $q$ is a prime power, the Weil conjectures predict the behavior of $\#X(\Ff_q)$\footnote{We use both $\# S$ and $|S|$ to denote the size of a set $S$.} as $q$ increases via encoding them in the Weil zeta function. The conjectures were proven by Dwork~\cite{dwork}, Grothendieck~\cite{grothendieck-L-function}, and Deligne~\cite{deligne-weil-I}. While the Weil conjectures give a precise formula for the number of rational points on varieties over finite fields, some weaker versions of the predictions were known long before the resolution of the Weil conjectures. Most notably, the Lang-Weil estimate~\cite{lang-weil-original} provides the asymptotic behavior of $\#X(\Ff_q)$ as $q$ goes to infinity.
Here is the precise statement of the classical Lang-Weil estimate.
\begin{theorem*}[Lang-Weil Estimate]\label{lang-weil-original}
    Given integers $n,d,r$, there is a constant $C=C(n,d,r)$ such that for every finite field $\Ff_q$ and every geometrically integral variety $X\subseteq\mathbb{P}^n$ over $\Ff_q$ of dimension $r$ and degree $d$, we have
    \begin{align*}
        |\#X(\Ff_q)-q^r|\leq (d-1)(d-2)q^{r-1/2}+Cq^{r-1}.
    \end{align*}
\end{theorem*}

For an algebraic variety $X$ defined over $\mathbb{Z}$ and each prime number $p$, we use $X_p$ to denote its fiber over the prime $p$. The above estimate gives us an estimate of $\#X_p(\Ff_q)$ where $q=p^m$ for some $m$ subject to the condition that $X_p$ is geometrically integral. Analogously, similar estimates for affine varieties can be obtained using this formula.
\subsection{Difference Lang-Weil Estimate}
One may wish to generalize this estimate to other settings, where one could study more complicated sets other than varieties, either obtained by first-order quantifiers or given in some expansions of fields. In the first direction, recall that a pseudofinite field is a perfect pseudo-algebraically closed (PAC)\footnote{A field $K$ is pseudo-algebraically closed if for any geometrically integral variety $X/K$, one has $X(K)\neq \varnothing$.} field with absolute Galois group $\widehat{\mathbb{Z}}$. The seminal work of Ax~\cite{Ax-fin} shows that this is equivalent to being a model of the asymptotic theory of the finite fields (a sentence is in the asymptotic theory of finite fields iff it is true for sufficiently large finite fields). Counting in finite fields admits natural extensions to pseudofinite fields, and the Lang-Weil estimate admits a generalization to definable sets in pseudofinite fields by the work of Chatzidakis, van den Dries, and Macintyre \cite{chatzidakis1992definable}.

In a parallel direction, model theory provides a uniform framework for looking at tame expansions of fields. The model theory of valued fields, differential fields, and difference fields are particularly well-studied expansions. Since our primary interest lies in counting in finite structures, the differential and valuation in these cases are trivial, and thus do not provide interesting information. However, the difference setting remains interesting.

Recall that a difference field is a field with a distinguished endomorphism. For example, a finite field with some iterate of the Frobenius is a finite difference field. Difference varieties can be thought of as generalizations of algebraic varieties, where the defining polynomials involve the given automorphism. It is natural to count rational points of difference algebraic varieties in a finite difference field. 
As in classical algebraic geometry \`a la Weil, one works in an ambient algebraically closed field $K$ and all varieties $V$ are identified with $V(K)$. In the difference setting, the theory ACFA, which was studied extensively in~\cite{ChHr99} and \cite{CHP}, provides the universal domains for difference algebraic geometry.  Many geometric notions admit difference variety analogues; in particular the {\em transformal dimension} of a difference variety $X$ is analogous to the dimension of an algebraic variety; for affine varieties it equals the maximum number of functions on the variety that are independent in the sense of satisfying no difference equation.  If the difference operator is specialized to a large Frobenius map $q$, one obtains an algebraic variety $M_q(X)$ , whose dimension is at most the transformal dimension of $X$.   A substantial contrast is that the algebraic closure of a difference field is not unique; see   Example~\ref{eg:dim-fail} for an effect of this on dimension.

ACFA is the model companion of the theory of difference fields. It is model complete and simple, a tameness property in the sense of Shelah~\cite{classification}. The work of EH~\cite{HrushovskiFrobenius} shows that ACFA is the asymptotic theory of $(\Ff_q^{alg},\frob_q)$ as $q \to \infty$. The key geometric input 
is a generalization of the Lang-Weil estimate to zero-dimensional difference varieties.   Here Lang-Weil is viewed geometrically, as estimating the size of a specific intersection of difference varieties of complementary dimension, namely an algebraic correspondence with the graph of Frobenius.   See also 
~\cite{ShuVar21}
for another proof.    

We are interested in a different generalization of Lang-Weil. Finite fields endowed with a Frobenius map are the `closed points' of the world of difference algebra, in the same way that finite fields play this role for ordinary commutative algebra.   We thus 
attempt to estimate the number of solutions of a difference variety 
in a finite difference field.  We achieve a nontrivial estimate in case the field is sufficiently large compared to the structural Frobenius map (see Theorem~B or Theorem~\ref{thm-finiteryCount}).

A geometric interpretation of this problem  - not mentioning rational points - would involve two commuting automorphisms,
realized by the $p$ and $q$-Frobenius.  When $q$ is a a high power of $p$, this theory interprets bounded arithmetic on the smaller field $\Ff_p$.   Perhaps our results hint at a possible theory relative to $\Ff_p$.   In any case  
 we will remain with the more arithmetic interpretation.

We are now able to state our question more precisely and give a brief sketch of the proof:

 Let $(D,\sigma)$ be a domain $D$ with an injective endomorphism $\sigma$ and $(P_i)_{i\leq N}$ be a finite collection of difference polynomials in variables $X_1,\ldots,X_n$ over $D$. In other words, $P_i$ lies in the polynomial ring $D[X_1,\ldots,X_n,X^\sigma_1,\ldots,X^{\sigma}_n,X^{\sigma^2}_1,\ldots]$. Like in the algebraic setting, we denote by $V$ the affine difference variety defined by $(P_i)_{i\leq N}$. Consider a difference ring homomorphism $\eta: D\to (\Ff_{p^t},\frob_q)$, where $q$ is a power of the prime $p$. Let $V^\eta$ denote the difference variety defined by the difference polynomials $(\eta(P_i))_{i\leq N}$.
We would like to count the number of rational points of $V^\eta$ in $(\Ff_{p^t},\frob_q)$. Namely, we need to estimate the size of the set of zeros of $(\eta(P_i))_{i\leq N}$ in $(\Ff_{p^t},\frob_{q})$, equivalently, the zero set in $\Ff_{p^t}$ of algebraic polynomials obtained from $\eta(P_i)$ by interpreting $X_i^{\sigma^m}$ as $ X_i^{q^m}$. Let us denote the resulting algebraic polynomial by $M_q(\eta(P_i))$ and the affine algebraic variety defined by $(M_q(\eta(P_i)))_{i\leq N}$ by $M_q(V^\eta)$. Our task is to estimate $\# M_q(V^\eta)(\Ff_{p^t})$ for large $t$, uniformly in terms of $p, q, t$ and $\eta$. 

The upper bound is easy (but optimal in this generality).  In the case of transformal dimension zero, we use the `trivial upper bound' of \cite{HrushovskiFrobenius}; for the case of a single difference polynomial $P_1$, it amounts simply to the degree of $M_q(P_1)$.  This bound applies to the number of solutions in $\Ff_p^{alg}$; we do not use here the additional information on rationality. At the opposite extreme, when the difference variety is the affine space, the number of solutions is $p^{tn}$.   The general case follows from these two by d\'evissage.

The challenge thus lies in the lower bound.   Let us mention a few issues around this.  First, $M_q(V^\eta)$ may be empty 
for infinitely many $q$, see Example~\ref{eg:dim-fail}.   In this example the problem is simply due to the incompleteness of ACFA$_0$. $V_0$ could simply be the two-point difference variety of a cube root of $3$ {\em in the fixed field}.   Then for $p=2 \mod 3$ and odd $t$, there can be no solution in $\Ff_{p^t}$.
Of course the same will be true for any higher dimensional $V$ admitting a morphism into this $V_0$.  

Secondly there arises the question of the number of irreducible components. Again this may occur due to a morphism onto a zero-transformal-dimesional variety $V_1$;
the number of components of $V$ in $\Ff_{p^t}$ will be at least the number of points of $V_1$.    This already accounts for the gap between our lower and upper bounds.  

There are also additional mechanisms for reducibility.  
For example, consider a difference polynomial $P:=X^\sigma X-Y^2$, it is not a product of two proper difference polynomials, while the corresponding algebraic polynomial $M_p(P):=X^{p+1}-Y^2$ is not irreducible for all $p\neq 2$.
This one can still be accounted for at the difference variety level, but we presume that rare cases will exist where $M_p(P)$ factorizes for non-uniform reasons, and do not know if this could happen for infinitely many $p$.

In case $M_p(V)$ has several components, the number of components rational over $F_{p^t}$ is the essential information; this aspects is already familiar in applications of the classical Lang-Weil.

The principal difficulty that we deal with here is, however, special to the difference setting; namely   $M_q(V)$ may be a variety of lower than expected dimension. This cannot happen for a single difference polynomial,  nor more generally if $V$ is cut out by $d'$ difference equations, where $d'$ is the codimension of $V$. But in general, $V$ may not be a `complete intersection' in this sense. We do not know if every difference variety can be fibered over a zero-transformal-dimensional one with 'complete intersection' fibers.  It is here that we need the theory of valued difference fields, allowing us to perturb solutions outside a smaller `singular' difference subvariety and thus show that $M_q(V)$ will locally have the expected dimension, regardless of the complete intersection question.

Our key result, Theorem~\ref{thm-equiDim},  resolves the above issues in the following terms:  There is a special difference subvariety $V_s$ of $V$ of strictly smaller transformal dimension, such that
for almost all $q$, 
$M_q((V\setminus V_s)^\eta)$ breaks into absolutely irreducible components of equal dimension $d$ for almost all $q$ and all $\eta$. 

Here is the precise statement.
\begin{thmA}
Let $V$ be a difference variety defined over a difference domain $D$ with transformal dimension $d>0$. Then there is a difference subvariety $V_s$ over $D$ of transformal dimension $<d$ and a difference domain $D'$ finitely generated over $D$ and contained in the fraction field $\Frac(D)$ and a constant $C$, such that for all $q>C$ and for all homomorphisms $\eta:D'\to (\Ff_p^{alg},\frob_q)$, the variety $M_q((V_s)^\eta)$ has algebraic dimension strictly smaller than $d$ and $M_q((V\setminus V_s)^\eta)$ is either empty or equidimensional of algebraic dimension $d$.
\end{thmA}

A brief remark about the proof of the theorem before we move on. To define $V_s$, we actually embed our difference field into a model of $\ACFA$ with a natural non-Archimedean topology coming from the work of Dor and EH~\cite{DoHr22} where such structures are called models of $\widetilde{\omega VFA}$. See Section~\ref{sec: vdf} for more details about this structure. More explicitly, the non-Archimedean topology is used in the following way. The subvariety $V_s$ is the closure of the points $p$ where an analogue of the implicit function theorem fails at $p$. Moreover, the proof yields that $V_s$ does not depend on the choice of embeddings into models of $\widetilde{\omega VFA}$. Though $V_s$ is a pure difference subvariety, our method uses non-difference algebraic data to find it. It is thus natural to ask:
\begin{question}
    Is there a purely difference algebraic description of $V_s$?
\end{question}

 Note that the proof of Theorem A and the above question require an interplay between the difference algebraic data and the non-Archimedean topology, it seems natural to ask if there exists a tame theory of analytic functions for difference fields.

With Theorem A in hand, it remains to find a smooth $\Ff_{p^t}$-point of $M_q((V\setminus V_s)^\eta)$. If such a point exists, the irreducible component (over $\Ff_{p^t}$) containing $x$ will be absolutely irreducible.

To find a smooth point, we extend the partial derivatives to difference polynomials and use the corresponding Jacobian criterion and generic smoothness. However, in characteristic $p>0$, generic smoothness fails as in the algebraic setting. One needs to modify the variety using the twisting reduction (see Definition~\ref{Def-Twist}), which is the difference analogue of the relative Frobenius (see~\cite[Section 1.2]{ez} for example). This procedure gives a stratification of $V$ via locally closed difference varieties which are smooth after twisting reduction. 

In conclusion, we establish the following dichotomy (see Theorem~\ref{thm-finiteryCount} for the final uniform version).

\begin{thmB}[Difference Lang-Weil Estimate]
    Let $X$ be a difference variety of transformal dimension $d>0$ defined over a difference domain $D$. Then there is a difference subvariety $X_{\xi}$ defined over $D$ of transformal dimension $<d$, constants $c,C>0$ and $D'\subseteq\Frac(D)$ finitely generated over $D$ such that for all $q>C$, for all homomorphisms $\eta:D'\to (\mathbb{F}_{p^t},\frob_q)$:
    \begin{enumerate}
    \item 
    Either there is a point $a\in (X^\eta\setminus X^\eta_{\xi})(\mathbb{F}_{p^t},\frob_q)$, and we have \[q^cp^{dt}\geq \#X^\eta(\mathbb{F}_{p^t},\frob_q)\geq p^{dt}-q^cp^{t(d-1/2)};\]
    \item 
    Or $X^\eta(\mathbb{F}_{p^t},\frob_q)\subseteq X^\eta_{\xi}(\mathbb{F}_{p^t},\frob_q)$, and $\# X^\eta(\mathbb{F}_{p^t},\frob_q)\leq q^cp^{t(d-1)}$.
\end{enumerate}
\end{thmB}

We remark that when $X$ is a pure algebraic variety, the proof of Theorem B gives that the $q^c$ in the count becomes a constant $c$ which does not depend on $q$. Hence, it is a version of the original Lang-Weil estimate where we do not require $X$ to be irreducible or geometrically irreducible. Indeed, as we mentioned before, requiring $X$ to be irreducible as a difference variety does not make $M_q(X)$ irreducible. However, our lower bound will be optimal if $M_q(X)$ happens to be irreducible for almost all $q$.

\begin{question}
    Is there a geometric or model-theoretic condition on the difference variety $X$ such that $M_q(X)$ is irreducible for almost all $q$?
\end{question}

Note that our count in $(\Ff_{p^t},\frob_q)$ is only meaningful when $p^t$ is significantly larger than $q$. Also the count is very rough, in the sense that when $X$ has transformal dimension $d$, and in case $X$ has a point $a$ in $(\Ff_{p^t},\frob_q)$ outside of $X_\xi$, then the estimate we give is a number between $(1-\epsilon)p^{td}$ and $q^cp^{td}$. To determine the precise coefficient of $p^{td}$, we need to count the number of irreducible components that are defined over $\Ff_{p^t}$. This is a highly non-trivial problem to solve, see Remark~\ref{rem: nb irreducible components} for further discussions.  Here we discuss the zero-dimensional case by way of illustration.

 When the difference variety $X$ has transformal dimension 0,
  the twisted Lang-Weil estimate (\cite{HrushovskiFrobenius},~\cite{ShuVar21}) tells us that there are finitely many rational numbers $(\mu_i)_{i\leq N}$ and natural numbers $(c_i)_{i\leq N}$ such that for all $\eta:D\to (\mathbb{F}_{p}^{alg},\frob_q)$,
\[|\#X^\eta(\mathbb{F}_{p}^{alg},\frob_q)-\mu_iq^{c_i}|\leq q^{c_i-1/2}\] holds for some $i\leq N$. Here $c_i$ is the total dimension of $X^\eta$.\footnote{Namely the maximum of transcendence degrees of difference function fields of irreducible components of $X^\eta$.}
But this estimate says nothing about the distribution of the points of $X$ among the various finite fields.  
 For example, consider the size of $\sigma(x)=x$, which is $|\Ff_{p^t}\cap \Ff_q|$, the precise number depends on the divisibility of $t$ and $\log_p q$. More generally, consider the difference polynomial $P(x)$ with leading term $\sigma^n(x^\ell)$. If $(\ell q^n)!| t$, then all solutions of $P(x)=0$ lie in $\Ff_{p^t}$, hence it has $\ell q^n$ solutions (counted with multiplicity). However, if $t$ is a prime, then $\Ff_{p^t}$ can contain at most $p$ solutions.  Additional, perhaps statistical results on this distribution would be very interesting.

Theorem B has an application to the model theory of difference fields. 

Consider $K=\prod_{i\to\mathcal{U}}(\Ff_{p_i^{n_i}},\frob_{p_i^{m_i}})$ an ultraproduct of finite difference fields over a non-principal ultrafilter $\mathcal{U}$ with $\lim_{i\to\infty}n_i/m_i=\infty$ and $\lim_{i\to \infty}p_i^{m_i}=\infty$. We prove that the coarse dimension $\pmb{\delta}$ (a dimension stemming from non-standard counting, see Section~\ref{sec:psdf} for details) of a quantifier-free type equals to the transformal transcendence degree of it, which resolves a conjecture by TZ in \cite{ZouDifference}.
\begin{thmC} Let $A$ be a difference subfield of $K$ and $r(x)=\qftp(a/A)$ where $a$ is a tuple in $K$ with transformal transcendence degree $d$ over $A$. Then $\pmb{\delta}(r)=d$.
\end{thmC}

\medskip

\textbf{Organization of the paper:} The paper is structured as follows. In Section~\ref{sec:prelim}, we review the basics of difference algebra, the model theory of difference fields, and valued difference fields. We also prove a useful result about the stability of transformal dimension under ultraproducts.

In Section~\ref{sec:loc_dim}, we work in $\widetilde{\omega VFA}$ and use its non-Archimedean topology to define $V_s$ uniformly. This definition helps us establish the analogue implicit function theorem in $V$ excluding $V_s$. This is crucial for the proof of Theorem A in Section~\ref{sec:equi_dim}.

Section~\ref{sec:twist_red} sets up the twisting reduction and proves a generic smoothness result for difference varieties after twisting reduction. Section~\ref{sec.Count} brings everything together and provides the proof of our main result, Theorem B (see Theorem~\ref{thm-finiteryCount}).

Finally, in Section~\ref{sec:psdf}, we explore the model-theoretic implications of our counting results.  Note that almost all results in this paper are stated in a uniform fashion, which invokes a substantial amount of technicalities to achieve.\\
\noindent\textbf{Acknowledgment:} MH and TZ was partially supported by the German Research Foundation (DFG) via HI 2004/1-1 (part of the French-German ANR-DFG project GeoMod) and under Germany's Excellence Strategy EXC 2044-390685587, `Mathematics M\"unster: Dynamics-Geometry-Structure'. JY was partially supported by the Fondation Sciences Math\'ematiques de Paris. We are grateful to Yuval Dor and Yatir Halevi for their insightful discussions.

\section{Preliminaries}\label{sec:prelim}
\textbf{Convention and Notation}: Throughout the paper, all rings are commutative with $1$ and morphisms are unital. 
For ultraproducts, let $I$ be an index set and $\mathcal{U}$ be an ultrafilter on $I$, and $(M_i)_{i\in I}$ be a family of $\mathcal{L}$-structures. We write the ultraproduct as $\prod_{i\to\mathcal{U}}M_i$, the ultraproduct of $X_i\subseteq M_i^n$ as $\prod_{i\to\mathcal{U}}X_i$ and a tuple from the ultraproduct as $(a_i)_{i\to\mathcal{U}}$. If $(\alpha_i)_{i\in I}$ is a sequence in $\mathbb{R}$, we denote by $\lim_{i\to\mathcal{U}}\alpha_i$ the unique $r\in\mathbb{R}\cup\{\pm\infty\}$ such that for any open interval $A$ containing $r$ (we regard $\infty\in(a,\infty)$, same for $-\infty$) one has $\{i\in I:\alpha_i\in A\}\in\mathcal{U}$.

\subsection{Difference Algebraic Preliminaries}\label{sec:pre1}
In this section, we recall some preliminaries on difference algebra. The main reference for this section is Cohn's book \cite{cohn}. 

Recall that a \emph{difference ring} $(R,\sigma)$ is a ring $R$ with a given injective morphism $\sigma:R\to R$. We occasionally refer to $R$ as a difference ring when $\sigma$ is understood from the context. For a prime power $q$, we use $\frob_q$ to denote the map $R\to R: x\mapsto x^q$.

Given a difference ring $(R,\sigma)$ and $n\in\mathbb{N}_{>0}$, let $R[x_1,\ldots,x_n]_\sigma$ denote the difference ring $R[x_1,\ldots,x_n,x_1^\sigma,\ldots x_n^\sigma,\ldots]$ with the natural (extension of) $\sigma$ on it. Note that we can evaluate difference polynomials just as one evaluates algebraic polynomials. Throughout the paper, `polynomial' means `difference polynomial', and we will call the usual polynomials \emph{algebraic polynomials}.

A difference ring $(K,\sigma)$ is a \emph{difference field} if the underlying ring $K$ is a field.  A difference field $(K,\sigma)$ is called \emph{inversive} if $\sigma$ is an automorphism, and it is called \emph{non-periodic} if $\Chara(K)=p$ and $\sigma^l\circ (\frob_p)^m \neq \mathrm{id}$ for any $(l,m)\neq (0,0)$.  Here, we set $\frob_p=\mathrm{id}$ if $p=0$.

By~\cite[Lemma II, p. 201]{cohn}, $(K,\sigma)$ is non-periodic if and only if there is no non-trivial polynomial that vanishes identically on $K$.

The inversive closure of a difference field $K$, i.e., the uniquely determined smallest inversive difference field containing $K$, will be denoted by $K^{\inv}$.

A \emph{morphism} of difference rings $f:(R_1,\sigma_1)\to (R_2,\sigma_2)$ is a morphism of rings that respects the difference operator. Given a difference field $K$, a \emph{$K$-difference 
 algebra} is a difference ring $R$ with a morphism $K\to R$. Morphisms of $K$-difference algebras are defined similarly. We will call $K$-difference 
 algebras/fields \emph{difference rings/fields over $K$} as well. For a difference field $L$ over $K$, we will denote the setup by $K\leq L$.

 For difference fields $K\leq L$ and $a\in L$ a tuple, let $K(a)_{\sigma}:=K(\sigma^k(a))_{k\in\Nn}$. Note that in case $K$ is inversive, we have $K(a)_\sigma^{\inv}=K(\sigma^i(a))_{i\in \mathbb{Z}}$. If $L=K(a)_\sigma$, then $a$ is a tuple of \emph{generators} of $L$ over $K$, and $L$ is said to be \emph{finitely generated over $K$} if the tuple of generators can be chosen to be finite. Likewise, if $(R,\sigma)\leq (S,\sigma)$ is an extension of difference rings and $a$ is a tuple from $S$, we set $R[a]_\sigma:=R[a,\sigma(a),\sigma^2(a),\ldots]$ and say that $S$ is \emph{finitely generated over $R$} if $S=R[a]_\sigma$ for some  finite $a$.
 
 Let $K\leq L=K(a)_\sigma$ be a finitely generated difference field extension. The \emph{limit degree} $\ld(L/K)_{\sigma}\in\mathbb{N}^{>0}\cup\{\infty\}$ is defined as \[\min\{[K(a,\sigma(a),\ldots,\sigma^{k+1}(a)):K(a,\sigma(a),\ldots,\sigma^{k}(a))]: k\in\mathbb{N}\},\]
and the \emph{reduced limit degree} $\rld(L/K)_\sigma$ is defined as \[\min\{[K(a,\sigma(a),\ldots,\sigma^{k+1}(a)):K(a,\sigma(a),\ldots,\sigma^{k}(a))]_s: k\in\mathbb{N}\},\] where $[E,F]_s$ is the separable degree of $E$ over $F$ for field extensions $F\leq E$.  By convention, the degree and the separable degree of a non-algebraic field extension $F/E$ is $\infty$. Note that $\ld(L/K)$ and $\rld(L/K)$ do not depend on the choice of generators~\cite[p. 135 and p. 140]{cohn}.

Given $(R,\sigma)$ a difference ring, a \emph{difference ideal} $\Sigma$ of $R$ is an ideal of the ring $R$ such that for $a\in \Sigma$, $\sigma(a)\in \Sigma$. A difference ideal $\Sigma$ is called \emph{prime} if it is a prime ideal, it is called \emph{reflexive} if $\sigma(a)\in\Sigma$ implies $a\in\Sigma$, and it is called \emph{perfect} if whenever $\prod_{i\leq N}\sigma^{n_i}(a)\in\Sigma$ for some $N$ and $n_i\geq 0$, then $a\in\Sigma$. In particular, a perfect difference ideal is radical. Note that any reflexive prime difference ideal is perfect.\footnote{A reflexive prime difference ideal is called a \emph{transformally prime ideal} in \cite{HrushovskiFrobenius}.} 

A \emph{difference domain} $R$ is a difference ring that is a domain. 

For a difference domain $(R,\sigma)$, $\sigma$ extends (uniquely) to $F=\Frac(R)$, we call $F=(F,\sigma)$ the \emph{fraction difference field}. Clearly, if $\sigma:R\to R$ is surjective, then $(F,\sigma)$ is inversive.

Let $(K,\sigma)$ be a difference field and $a$ be a tuple of length $n$ in a $K$-difference field $L$. Define $I(a/K)=\{P\in K[x_1,\ldots,x_n]_\sigma: P(a)=0\}$. Note that $I(a/K)$ is a reflexive prime difference ideal. 

Conversely, any proper reflexive prime difference ideal of $K[x_1,\ldots,x_n]_\sigma$ is of the form $I(b/K)$ for some $n$-tuple $b$ in some difference field extension of $K$ (see \cite[Theorem~VIII on p.~77]{cohn}).

Now we define some terminology for difference polynomial rings. We use $\mathbb{N}[\sigma]$ to denote the commutative semiring generated by $\sigma$ over $\mathbb{N}$. We equip it with the natural ordering such that $\sigma>n$ for all $n\in \mathbb{N}$. 

Let $Q\in K[x_1,\ldots,x_n]_{\sigma}$ be given of the form $Q=\sum_{j\leq J}a_j\prod_{1\leq i\leq n}x_i^{\mu_{i,j}}$ with $a_j\in K\setminus\{0\}$ and $\mu_{i,j}\in\Nn[\sigma]$ such that if $j\neq j'$, then $(\mu_{1,j},\ldots,\mu_{n,j})\neq (\mu_{1,j'},\ldots,\mu_{n,j'})$. The \emph{total (difference) degree} of $Q$ is defined to be $\max\{\sum_{1\leq i\leq n}\mu_{i,j}:j\in J\}\in \mathbb{N}[\sigma]$.

Next, we define a natural derivation on difference polynomials. Let $K$ be a difference field. We define the partial derivative $\frac{\partial }{\partial x_i}$ on $K[x_1,\ldots,x_n]_{\sigma}$ for $1\leq i\leq n$ as the unique $K$-linear operator satisfying the Leibniz rule such that
    \begin{itemize}
        \item $\frac{\partial x_i}{\partial x_i}=1$ and $\frac{\partial x_j}{\partial x_i}=0$ for $i\neq j$;
        \item
        $\frac{\partial \left(x_j^{\sigma^k}\right)}{\partial x_i}=0$ for all $k\geq 1$ and $1\leq i,j\leq n$. 

\end{itemize}
Lastly, let $K$ be an inversive difference field of characteristic $p$. Denote by $K[x_1,\ldots,x_n]^{\pm}_{\sigma}$ the difference ring $K[ ((\frob_p)^s(x_i))^{\sigma^z}:z,s\in\mathbb{Z},1\leq i\leq n]$ 
with the obvious identifications, i.e., identify it as a subring of the perfect inversive closure of $\Frac(K[x_1,\ldots,x_n]_{\sigma})$.

The following is the difference version of Noetherianity.
\begin{definition}
     A difference ring is \emph{Ritt} if it satisfies the ascending chain condition for perfect difference ideals, i.e., for any sequence $\Sigma_1\subseteq \Sigma_2\subseteq\cdots$ of perfect difference ideals, there is some $k$ such that $\Sigma_n=\Sigma_k$ for all $n\geq k$. 
\end{definition}
The following result is the difference analogue of Hilbert's Basis Theorem (see \cite[Theorem~V, p.~89]{cohn} for a proof).
 \begin{fact}[\mbox{Ritt-Raudenbush Basis Theorem}]\label{RittRaud}
        A polynomial difference ring $R[x_1,\ldots,x_n]_\sigma$ over a Ritt difference ring $R$ is a Ritt difference ring.
    \end{fact}
    \begin{remark}\label{Rm-Z-Noetherian}
        If $R$ is a Noetherian ring with any injective endomorphism $\sigma$ on $R$, then $(R,\sigma)$ is a Ritt difference ring. Hence, $R[x_1,\ldots,x_n]_\sigma$ is Ritt. In particular, if $K$ is a difference field, then the difference polynomial ring $K[x_1,\ldots,x_n]_\sigma$ is a Ritt difference ring.
    \end{remark}

    \begin{definition} 
    Let $K=(K,\sigma)$ be a difference field, for each $n\in\mathbb{N}$, we denote the \emph{difference affine $n$-space} $\Aa^n_{(K,\sigma)}$ to be the topological space whose points consist of reflexive prime difference ideals of the difference polynomial ring $K[x_1,\ldots,x_n]_\sigma$. It is equipped with the \emph{Zariski-Cohn topology} whose
    closed sets are of the form
    \[V(S):=\{p\in\Aa^n_{(K,\sigma)},S\subseteq p\},\]
    where $S\subseteq K[x_1,\ldots,x_n]_\sigma$ is any subset.
    
    Note that $K^n$ can be identified with 
 a subset of $\Aa^n_{(K,\sigma)}$ via $a\mapsto I(a/K)$. We endow $K^n$ with the induced topology, which will also be called the Zariski-Cohn topology (on $K^n$).
    \end{definition}
    \begin{remark}

        Given any $S\subseteq K[x_1,\ldots,x_n]_\sigma$, let $\Sigma_{S}$ be the perfect difference ideal generated by $S$, i.e., the intersection of all perfect difference ideals containing $S$. Then $V(S)=V(\Sigma_S)$. On the other hand, if $\Sigma$ is a perfect difference ideal, by~\cite[p.~88]{cohn}, \[\Sigma=\bigcap_{p\in\Aa^n_{(K,\sigma)}, \Sigma\subseteq p}p.\] Hence $I(V(S)):=\bigcap_{p\in V(S)}p=\Sigma_S$. In other words, there is an order-reversing one-to-one correspondence between closed subsets of $\Aa^n_{(K,\sigma)}$ and perfect difference ideals in $K[x_1,\ldots,x_n]_\sigma$.  Since $K[x_1,\ldots,x_n]_\sigma$ is Ritt by Remark~\ref{Rm-Z-Noetherian}, $\Aa^n_{(K,\sigma)}$ is a Noetherian topological space.
  
\end{remark}
    \begin{definition}
      \begin{itemize}  
      \item Let $K=(K,\sigma)$ be a difference field. An \emph{(affine)\footnote{In this paper, we only consider affine difference varieties.} difference variety} $W$ over $K$ is a closed subset of $\Aa^n_{(K,\sigma)}$ for some $n$. 
        We will often write $W\subseteq\Aa^n$ when $K$ is understood from the context. 
       \item  The Zariski-Cohn topologies on $W$ and $W(K)$ are the subspace topologies induced by $\Aa^n_{(K,\sigma)}$.
     We say $U\subseteq \Aa^n_{(K,\sigma)}$ is a \emph{locally closed subvariety} of $\Aa^n_{(K,\sigma)}$ if $U$ is a locally closed subset of $\Aa^n_{(K,\sigma)}$. Similarly, $U$ is an \emph{open subvariety} of $V$ if $U$ is open in $V$.  A difference variety $V$ is called \emph{irreducible} (over $K$) if it is irreducible as a topological space.
        
\item  For any difference variety $W$, $I(W)$ is finitely generated as a perfect difference ideal in the sense that $I(W)=\Sigma_S$ for some finite set $S$ which we call a set of \emph{defining polynomials} for $W$.
Conversely, given a finite set of polynomials $\{P_1,\ldots,P_k\}\subseteq K[x_1,\ldots,x_n]_\sigma$, we say that $W:=V(\{P_1,\ldots,P_k\})$ is the difference variety defined by $P_1,\ldots,P_k$.

         \end{itemize}
    \end{definition}
\begin{definition}\label{Def:overD}
     Let $D=(D,\sigma)$ be a difference domain, with fraction field $K$. A difference variety $V\subseteq \Aa^n$ over $K$ is said to be \emph{defined over $D$}, if $\Sigma:=I(V)\cap D[x_1,\cdots,x_n]_\sigma$ is finitely generated as a perfect difference ideal in $D[x_1,\cdots,x_n]_\sigma$. And in this case, we say that $V$ is defined over $D$ with defining polynomials $P_1,\ldots,P_k$ if $\{P_1,\ldots,P_k\}$ is a set of generators of $\Sigma$ in $D[x_1,\cdots,x_n]_\sigma$.
\end{definition}

\begin{remark}
     If $D$ is finitely generated over $\mathbb{Z}$, then $D[x_1,\ldots,x_n]_\sigma$ is a Ritt difference ring.  For any difference variety $V$ defined over the fraction field of $D$, the ideal $I(V)\cap D[x_1,\cdots,x_n]_\sigma$ is thus finitely generated as a perfect difference ideal, and so there is a finite set of generators $\{P_1,\ldots, P_k\}\subseteq D[x_1,\ldots,x_n]_\sigma$ such that $V$ is defined over $D$ by $P_1,\ldots,P_k$.

     Note, however, that it is not the case that every set of generators for $I(V)$ from $D[x_1,\ldots,x_n]_\sigma$ is a set of generators for $I(V)\cap D[x_1,\ldots,x_n]_\sigma$.
\end{remark}

\begin{fact}[{\cite[Theorem I,II, p. 111]{cohn}}]
       A difference variety $W\subseteq \Aa^n_{(K,\sigma)}$ is irreducible if and only if $I(W)$ is a reflexive prime difference ideal.
\end{fact}

\begin{definition}
 If $W\subseteq \Aa^n_{(K,\sigma)}$ is an irreducible difference variety, we define its \emph{function field} $K(W)_\sigma$ to be the fraction difference field of $K[x_1,\ldots,x_n]_\sigma/I(W)$.    
\end{definition}

\begin{definition}
    Let $K\leq L$ be a difference field extension. A family $\{a_i\}_{i\in I}$ of elements from $L$ is called \emph{transformally dependent over $K$} if there are pairwise distinct $i_1,\ldots, i_n$ from $I$ and a nonzero polynomial $P\in K[x_1,\ldots,x_n]_\sigma$ such that $P(a_{i_1},\ldots, a_{i_n})=0$. Otherwise, $\{a_i\}_{i\in I}$ is called \emph{transformally independent} over $K$. An element $a\in L$ is said to be \emph{transformally algebraic} over $K$ if $\{a\}$ is transformally dependent over $K$, and \emph{transformally transcendental} over $K$ otherwise. 

    A transformally independent set $S\subseteq L$ over $K$ is called a \emph{transformal transcendence basis} of $L$ over $K$, if $S$ is a maximal transformally independent set in $L$ over $K$.

    The \emph{transformal transcendence degree} of $L$ over $K$, denoted by $\trfdeg(L/K)$, is the cardinality of a transformal transcendence basis of $L$ over $K$.\footnote{Transformal algebraic closure satisfies Steinitz' Exchange Lemma, so it defines a pregeometry in any difference field and the notion of transformal transcendence degree is well-defined.}
\end{definition}

\begin{definition}
Let $W\subseteq \Aa^n$ be an irreducible difference variety over $K$, and let 
 $K\leq L$ be a difference field extension with $a=(a_1,\ldots,a_n)$ a tuple in $L$. Then $a$ is called a \emph{generic point} of $W$ over $K$ if $I(a/K)=I(W)$. In this case, $K(a)_{\sigma}$ is naturally isomorphic to the function field $K(W)_{\sigma}$.

For an irreducible difference variety $V$ over $K$, its \emph{transformal dimension} $\tdim_K(V)$ is defined to be $\trfdeg(K(V)_\sigma/K)$. If $V$ is a difference variety (over $K$), its transformal dimension  $\tdim_K(V)$ is defined to be the maximal transformal dimension of the irreducible components of $V$ (over $K$). Let $V$ be defined over some difference domain $D$. Then we set $\tdim_D(V):=\tdim_K(V)$ where $K=\Frac(D)$.
\end{definition}

To avoid any confusion, the dimension of an algebraic variety $X$ will be called its \emph{algebraic dimension}, denoted by $\adim(X)$.

\begin{definition}
  Let $W\subseteq \Aa^n$ be a difference variety over $K$ with defining polynomials $P_1,\ldots,P_k$. Let $L\geq K$. The \emph{set of $L$-rational points of $W$}, denoted by $W(L)$, is defined as \[W(L):=\left\{a\in L^n, \bigwedge_{1\leq i\leq k}P_i(a)=0\right\}.\]

  An embedding $f:L\to L'$ of difference fields over $K$ induces a natural map $W(f):W(L)\to W(L')$.

\end{definition}

Note that $W(L)$ does not depend on the set of defining polynomials.

\begin{definition}
    Let $E\geq F$ be a difference field extension.  The \emph{core  of $E$ over $F$}, denoted by $\mbox{Core}(E/F)$, is  the difference subfield which consists of all elements $a\in E$ which are separable algebraic over $F$ and satisfy $\ld(F(a)_\sigma/F)=1$.
\end{definition}

The following fact is a consequence of \cite[Lemma~2.8 and Lemma~2.9]{ChHr99}.\footnote{See also \cite[page 5]{zoeNotes}.}

\begin{fact}[\mbox{Babbitt's Theorem}]\label{fact:babbit}
Let $(E,\sigma)$ be a difference field. Then two extensions of $\sigma$ to $E^{\text{alg}}$ are $E$-isomorphic if and only if their restrictions to $\mbox{Core}(E^{\text{alg}}/E)$ are isomorphic.
\end{fact}

\bigskip

\subsection{ACFA}
In this section, we will summarize some results about the model companion of the theory of difference fields, ACFA. The main reference is \cite{ChHr99} by Chatzidakis and EH. Unless stated otherwise, in this section all difference fields are assumed to be inversive.

Let $\mathcal{L}_\sigma$ be the language of difference rings, namely the language of rings augmented by one function symbol $\sigma$.

If $q$ is a power of some prime $p$,  we denote by $K_q$ the difference field $(\mathbb{F}_p^{alg},\frob_q)$. Moreover, we denote by $Q$ the set of all prime powers.

For items (1-5) of the following fact, see \cite[1.1--1.7 and  Corollary 2.11]{ChHr99}. Item (6) is \cite[Theorem 1.4]{HrushovskiFrobenius}, which depends on the twisted Lang-Weil estimate in \cite{HrushovskiFrobenius}. For a recent purely geometric proof of this estimate, see \cite{ShuVar21}.

\begin{fact}\label{fact:ACFA}
    \begin{enumerate}
        \item The theory of difference fields in $\mathcal{L}_\sigma$ admits a model companion, called ACFA.
      \item If $(E,\sigma)\models ACFA$ and $B\subseteq E$ then the algebraic closure of $B$, denoted by $\acl_\sigma(B)$, is given by the field-theoretic algebraic closure of the field generated by $\bigcup_{z\in\Zz}\sigma^z(B)$.
\item If $(E_i,\sigma_i)\models ACFA$ and $B_i\subseteq E_i$ for $i=1,2$, then $B_1\equiv_{\mathcal{L}_\sigma} B_2$ if and only if there is an $\mathcal{L}_\sigma$-isomorphism from $\acl_\sigma(B_1)$ to $\acl_\sigma(B_2)$ sending $B_1$ to $B_2$.

Equivalently, if $K=(K,\sigma)$ is a difference field with $K$ algebraically closed and $K\subseteq E_i\models ACFA$ for $i=1,2$, then $E_1\equiv_{K}E_2$.
        \item Every $\mathcal{L}_\sigma$-formula $\phi(x)$ in a tuple $x$ is equivalent in ACFA to a formula of the form $\exists y\psi(x,y)$, such that $\psi(x,y)$ is quantifier-free, $|y|=1$ and there is $m\in\mathbb{N}$ such that $\psi(a,b)$ implies that $b$ is algebraic (in the field-theoretic sense) over the field $A$ generated by $a,\sigma(a),\ldots,\sigma^m(a)$. Moreover, the degree of $b$ over $A$ is bounded uniformly.
        \item
        ACFA has a canonical complete definable 1-type $p_{gen}(x)$, determined by the formulas $P(x)\neq 0$ for all nonzero polynomials. More generally, for each $n$, there is a unique generic $n$-type, $p_{gen}^n(x_1,\ldots,x_n)$ generated by $P(x_1,\ldots,x_n)\neq 0$ for all nonzero  polynomials. Equivalently, for any difference subfield $K$ in a monster model of ACFA, $p_{gen}^n|_K=\otimes^n p_{gen}|_K=\tp(a_1,\ldots,a_n/K)$ where $a_i\models p_{gen}|_{Ka_1,\ldots,a_{i-1}}$ for $1\leq i\leq n$.
        \item
        ACFA is the asymptotic theory of $K_q$. Namely, 
        $$\mbox{ACFA}=\{\phi \text{ an } \mathcal{L}_\sigma\text{-sentence }: K_q\models \phi\text{ for all }q\in Q\text{ with }q\gg 0.\}$$ In particular, if $\mathcal{U}$ is a non-principal ultrafilter on $Q$, then $\prod_{q\to \mathcal{U}}K_q\models ACFA$.
    \end{enumerate} 
\end{fact}

Let $E\models\mbox{ACFA}$ and $\monster\succcurlyeq E$ be $|E|^+$-saturated. Then for any $\mathcal{L}_\sigma(E)$-formula $\phi(x_1,\ldots,x_n)$ and corresponding definable set $X=\phi(E)\subseteq E^n$, we set 
$$\dim_{\mathrm{ACFA}}(X):=\max\{\trfdeg(E(a)_\sigma/E): a\in\phi(\monster)\}.$$

\begin{remark}
    Let $X\subseteq E^n$ be definable in $E\models\mathrm{ACFA}$, then $\dim_{\mathrm{ACFA}}(X)$ is the maximal $d$ such that there is a projection $\pi:E^n\to E^d$ with $\pi(X)$ containing $p^{d}_{gen}$. It follows that $\dim_{\mathrm{ACFA}}$ is uniformly definable across all completions of $\mathrm{ACFA}$.
\end{remark}

The following is immediate from Fact~\ref{fact:ACFA}(1).
\begin{fact}\label{lem:ACFA_comp_dim}
    Let $V$ be a difference variety defined over $K$. Then the transformal dimension of $V$ over $K$ is \[\max\{\dim_{\mathrm{ACFA}}(V(E)): K\leq E\models\mbox{ACFA}\}.\]
\end{fact}
The next lemma is presumably well-known, we include it due to the lack of references.

\begin{lemma}\label{ultratrd}
Let $K=(K,\tau):=\prod_{i\rightarrow\mathcal{U}}(K_i,\tau_i)$ be an ultraproduct of difference fields.
Let $V=\prod_{i\to\mathcal{U}}V_i\subseteq \Aa^n_{(K,\tau)}$ be a difference variety over $K$. Then $$\tdim_K(V)=\lim_{i\to\mathcal{U}}\tdim_{K_i}(V_i).$$
\end{lemma}

\begin{proof}
Set $d:=\tdim_K(V)$, $d_i:=\tdim_{K_i}(V_i)$ and $d':=\lim_{i\to\mathcal{U}}d_i$. We first prove that $d'\leq d$. By  Fact~\ref{lem:ACFA_comp_dim}, for any $i$ there is an embedding $f_i:(K_i,\tau_i)\hookrightarrow (E_i,\tau_i')\models \mbox{ACFA}$ such that $\dim_{\mathrm{ACFA}}(V_i(E_i))=d_i$. Moreover, $(K,\tau)\hookrightarrow\prod_{i\rightarrow\mathcal{U}}(E_i,\tau_i')=:(\widetilde{E},\tilde{\tau})\models\mbox{ACFA}$. As $\dim_{\mathrm{ACFA}}$ is definable (uniformly in all completions of $\mbox{ACFA}$), $\dim_{\mathrm{ACFA}}(V(\widetilde{E}))=d'$. Therefore, $d'\leq d$ by Fact~\ref{lem:ACFA_comp_dim}.

Now we show that $d\leq d'$. Let $f:(K,\tau)\hookrightarrow (E,\tau')$ be an embedding such that $(E,\tau')\models \mbox{ACFA}$ and $\dim_{\mathrm{ACFA}}(V(E))=d$. Suppose $V$ is defined by $\varphi(x,c)$ for $c=(c_i)_{i\to \mathcal{U}}$ in $K$ where $\varphi(x,z)$ is an $\mathcal{L}_\sigma$-formula and $\varphi(x,c_i)$ defines $V_i$ for each $i$. Since transformal dimension is definable in $\mbox{ACFA}$, there is an $\mathcal{L}_\sigma$-formula $\phi(z)$ such that for any $E'\models \mathrm{ACFA}$ and $e\in E'$ with $E'\models \phi(e)$, the formula $\varphi(x,e)$ defines a difference variety $V_e$ such that $\dim_{\mathrm{ACFA}}(V_e(E'))=d$. 

Applying Fact~\ref{fact:ACFA}(4), we may assume that $\phi(z)$ is bounded existential, i.e., of the form $\exists y\psi(z,y)$ with the properties stated in Fact~\ref{fact:ACFA}(4).

It follows in particular that $(K^{\text{alg}},\tau'\upharpoonright_{K^{\text{alg}}})\models \exists y\psi(c,y)$. Set $K_{cor}:=\mbox{Core}(K^{\text{alg}}/K)$. Let $\mbox{ACF}_{\sigma}$ denote the $\mathcal{L}_\sigma$-theory of algebraically closed fields with an automorphism. Then
Babbitt's theorem~(see Fact~\ref{fact:babbit}) implies

$$\mbox{ACF}_{\sigma}\cup\mathrm{qftp}(K_{cor},\tau'\upharpoonright_{K_{cor}})\models \exists y\psi(c,y).$$ 

By compactness, there is a quantifier-free sentence $\xi(a)$ with $a$ a tuple containing $c$ in $K_{cor}$ such that $\mbox{ACF}_{\sigma}\cup\{\xi(a)\}\models \exists y\psi(c,y).$ Since $a\in K_{cor}$, there is a natural number $n$ such that the smallest field extension of $K$ containing $a$ that is closed under $\tau'$ is of degree $n$. Thus there is a first-order statement expressing the following: \emph{There is a degree $n$ difference field extension of $K$ containing $u$ such that $\xi(u)$ holds.}
By \L o\'{s}'s Theorem, for $\mathcal{U}$-many $i$'s there is $(K_i,\tau_i)\leq (K_i',\tau_i')$ where $K_i'$ is a degree $n$ difference field extension of $K_i$ and $c_i\subseteq a_i\subseteq K_i'$ such that $(K_i',\tau_i')\models \xi(a_i)$. 

Then for any embedding of $(K_i',\tau_i')$ into some $E_i\models\mbox{ACFA}$, we have $E_i\models\exists y\psi(c_i,y)$ as $E_i\models\xi(a_i)$. Thus,  $\dim_{\mathrm{ACFA}}V_i(E_i)=d$ and thus $d_i\geq d$ for $\mathcal{U}$-many $i$'s as desired.\qedhere

\end{proof}

\bigskip

\subsection{Valued difference fields}\label{sec: vdf}
In this section, we present the necessary material on valued difference fields which we will need for our purposes. The results in arbitrary characteristic are contained in the preprint \cite{DoHr22} by Dor and EH.  In characteristic 0, everything we will use is contained in Durhan's work \cite{Azg10} (see also \cite{ChHi14}).

An \emph{ordered difference group }is a structure of the form $\langle \Gamma,0,+,<,\sigma\rangle$, where $\langle\Gamma,0 ,+,<\rangle$ is an ordered 
abelian group and $\sigma$ is an automorphism of $\langle\Gamma,0 ,+,<\rangle$. The automorphism $\sigma$ is called \emph{$\omega$-increasing} 
if $\sigma(\gamma)> n\gamma$ for all $\gamma\in\Gamma_{>0}$ and all natural numbers $n$. We treat ordered difference groups as first-order structures in the 
language $\mathcal{L}_{OGA}=\{0,+,<,\sigma\}$; the theory of $\omega$-increasing ordered difference groups may be axiomatised in $\mathcal{L}_{OGA}$, and we denote it by 
$\omega OGA$. Any model of $\omega OGA$ is naturally an ordered $\Zz[\sigma]$-module, where $\Zz[\sigma]$ is the ordered ring of polynomials 
in the indeterminate $\sigma$ with $\sigma\gg1$. Divisible modules of such kind correspond to ordered vector spaces 
over the ordered fraction field $\Qq(\sigma)$ of $\Zz[\sigma]$. The theory of non-trivial divisible ordered $\Zz[\sigma]$-modules will be denoted by $\widetilde{\omega OGA}$. It is easy to see that $\widetilde{\omega OGA}$ is the model-completion of $\omega OGA$.

Now let $(K,\val)$ be a valued field, with value group $\Gamma_K$. We will denote by $\val:K\rightarrow\Gamma_K\cup\{\infty\}$ the valuation map. Let $\Oo_K=\{x\in K\mid \val(x)\geq0\}$ be the corresponding valuation ring, with unique maximal ideal $\mm_K=\{x\in K\mid \val(x)>0\}$ and residue field $k_K=\Oo_K/\mm_K$, and let $\mathrm{res}:\Oo_K\rightarrow k_K$ be the residue map.

A \emph{valued difference field }is a valued field $(K,\val)$ together with a distinguished automorphism 
of $(K,\val)$, i.e., a field automorphism $\sigma$ of $K$ satisfying $\sigma(\Oo_K)=\Oo_K$. Note that $\sigma$ induces 
an automorphism $\overline{\sigma}$ of the residue field, making it an inversive difference field. Similarly, $\sigma$ induces an 
automorphism $\sigma_\Gamma$ of the value group, making it an ordered difference group. 

We will treat valued difference fields in the three-sorted language $\mathcal{L}_{\rf,\vg,\sigma}$, consisting of 
\begin{itemize}
\item the language of difference rings $\mathcal{L}_{\vf}=\{0,1,+,\times,\sigma\}$ on the valued field sort denoted by $\vf$;
\item (a copy of) the language of difference rings  $\mathcal{L}_{\rf}=\{0,1,+,\times,\overline{\sigma}\}$ on the residue field sort denoted by $\rf$;
\item the language of ordered difference groups (with an additional constant for $\infty$) given by $\{0,<,\infty,+,\sigma_{\vg}\}$ on the value group sort denoted by $\vg$, and
\item the functions $\val:\vf\rightarrow\vg$ and $\res:\vf\rightarrow\rf$ between the sorts. (When considering a valued field as an $\mathcal{L}_{\rf,\vg,\sigma}$-structure, 
we make the function $\res$ total by sending elements of negative valuation to $0\in\rf$.)
\end{itemize}

\begin{definition}
A valued difference field $\mathcal{K}=(K,\Gamma_K,k_K,\val,\sigma)$ is called \emph{contractive} if its value group $\Gamma_K$ is an $\omega$-increasing ordered difference group. The $\mathcal{L}_{\rf,\vg,\sigma}$-theory of contractive valued difference fields will be denoted by $\omega VFA$.
\end{definition}

\begin{example}\label{E:Hahn}
\begin{enumerate}
    \item Let $(k,\overline{\sigma})$ be a difference field and let $(\Gamma,\sigma_\Gamma)$ be an $\omega$-increasing ordered difference group. On the Hahn series field $K:=k((t^\Gamma))$ one may define an automorphism $\sigma$, setting $\sigma(\sum_\gamma c_\gamma t^\gamma):=\sum_\gamma \overline{\sigma}(c_\gamma)t^{\sigma_\Gamma(\gamma)}$. Then $(K,\Gamma,k,\val_t,\sigma)$ is a contractive valued difference field with residue field $(k,\overline{\sigma})$ and value group $(\Gamma,\sigma_\Gamma)$.
    \item For $q=p^n\in Q$, consider the valued difference field $\mathcal{L}_q:=(\Ff_q(t)^{alg},\Qq,\Ff_q^{alg},\val_t,\frob_q)$, where $\val_t$ denotes an extension of the $t$-adic valuation on $\Ff_q(t)$ to the algebraic closure. Let $\mathcal{U}$ be a non-principal ultrafilter on $Q$. Then $\prod_{q\to \mathcal{U}}\mathcal{L}_q\models \omega VFA$.
\end{enumerate}
\end{example}

We now state some of the main results from \cite{DoHr22}.

\begin{fact}[{\cite[Thm~9.10, Thm~9.14 and Thm~9.22]{DoHr22}}]\mbox{}\label{omegaVFA}
\begin{enumerate}
    \item The theory $\omega VFA$ admits a model-companion $\widetilde{\omega VFA}$.
    \item Let $\mathcal{F}=(F,\Gamma_F,k_F,\sigma)\models\omega VFA$ with $F=F^{alg}$, and let $F\subseteq E_i\models\widetilde{\omega VFA}$ for $i=1,2$. Then $E_1\equiv_F E_2$.
    \item In any model of $\widetilde{\omega VFA}$, the following holds:
    \begin{itemize}
        \item The residue field $\rf$ is stably embedded, with induced structure a pure model of ACFA. 
        \item The value group $\vg$ is stably embedded, with induced structure a pure model of $\widetilde{\omega OGA}$.
        \item $\rf$ and $\vg$ are orthogonal.
    \end{itemize}
    \item $\widetilde{\omega VFA}$ is the asymptotic theory of $\mathcal{L}_q$. Namely, 
    $$\widetilde{\omega VFA}=\{\phi\text{ a sentence in $\mathcal{L}_{\rf,\vg,\sigma}$}\mid \mathcal{L}_q\models \phi\text{ for all $q\in Q$ with $q\gg 0$}\}.$$ In particular, if $\mathcal{U}$ is a non-principal ultrafilter on $Q$, then $\prod_{q\to \mathcal{U}}\mathcal{L}_q\models \widetilde{\omega VFA}$. It also follows that $\widetilde{\omega VFA}\supseteq  ACFA$.\qedhere
\end{enumerate}
\end{fact}

An explicit axiomatisation of $\widetilde{\omega VFA}$ is given in \cite[Definition~0.8]{DoHr22}. We do not state the axiomatisation, as we will not make use of it.

\begin{definition}
In $\widetilde{\omega VFA}$, the \emph{generic type} $p_{\Oo}(x)$ of $\Oo$ is the (a priori only partial) global definable type of an element $x\in\Oo$ such that 
$\res(x)$ is transformally transcendental.  
\end{definition}

\begin{fact}
In any completion of $\widetilde{\omega VFA}$, the type $p_{\Oo}(x)$ is a complete definable type. Moreover, $p_{\Oo}(x)$ commutes with itself, i.e., $p_{\Oo}(x)\otimes p_{\Oo}(y)=p_{\Oo}(y)\otimes p_{\Oo}(x)$.
\end{fact}

\begin{proof}
Let $\mathcal{K}=(K,\Gamma,k,v,\sigma)\models \widetilde{\omega VFA}$ and $a,a'\models p_{\Oo}|K$. We need to show that $\tp(a/K)=\tp(a'/K)$. By Fact~\ref{omegaVFA}(2), it is enough to show that 
there is an isomorphism of valued difference fields $(K(a)_{\sigma}^{alg},v,\sigma)\cong (K(a')_\sigma^{alg},v,\sigma)$ over $K$ sending $a$ to $a'$. 

By the uniqueness of the Gauss valuation and induction, we have a natural isomorphism $f:(K(a)_{\sigma},v,\sigma)\cong_K (K(a')_\sigma,v,\sigma)$ sending $a$ to $a'$. By the universal property of the henselization, this $f$ extends (uniquely) to an isomorphism $f^{h}:(K(a)_\sigma^{h},v,\sigma)\cong_K (K(a')_\sigma^{h},v,\sigma)$. By \cite[Proposition~6.16]{DoHr22}, $\mbox{Core}(K(a)_\sigma^{alg}/K(a)_\sigma^{h})=K(a)_\sigma^{h}$. Hence, $f^{h}$ extends to an isomorphism of difference fields  $f^{alg}:(K(a)_\sigma^{alg},\sigma)\cong (K(a')_\sigma^{alg},\sigma)$ by Babbitt's Theorem. As $K(a')_\sigma^{h}$ is henselian, any such $f^{alg}$ automatically respects the valuation, and so $f^{alg}$ is an isomorphism of valued difference fields, as required. 

The moreover part is clear.
\end{proof}

\begin{corollary}\label{cor-uniqueGeneric}
Let $F=(F,\sigma)$ be a difference subfield of a model of $\widetilde{\omega VFA}$.
    Let $p(x_1,\ldots,x_d)$ be the partial type (in the language of valued difference fields) over $F$ expressing that $x_i\in\Oo$ for all $i$ and that the tuple $(\res(x_1),\ldots,\res(x_d))$ is transformally independent over $\res(F)$. Then $p(x_1,\ldots,x_d)$ is a complete type over $F$.
    Indeed $p(x_1,\ldots,x_d)= \otimes^d p_{\Oo}|_F$.
\end{corollary}

Recall that if $K$ is a valued field and $X\subseteq K^n$, then the \emph{topological dimension} of $X$, denoted by $\mathrm{top.}\dim(X)$, is defined to be the maximal $d\leq n$ such that there is a projection $\pi:K^n\rightarrow K^d$ with $\mathrm{int}_{\val}(\pi(X))\neq\emptyset$, where $\mathrm{int}_{\val}$ denotes the interior in the valuation topology.

\begin{lemma}\label{topdim=trfdim}
Let $\mathcal{K}=(K,\Gamma,k,v,\sigma)\models \widetilde{\omega VFA}$ and $X\subseteq K^n$ be definable (with parameters) in the language of difference rings. Then $\mathrm{top.}\dim(X)=\dim_{\mathrm{ACFA}}(X)$.
\end{lemma}

\begin{proof}
Note that $\dim_{\mathrm{ACFA}}(X)$ is given by the maximal $d$ such that there is a projection $\pi:K^n\rightarrow K^d$ with the property that the ACFA-generic type $p^d_{gen}|_K$ concentrates on $\pi(X)$. 

It is thus enough to show that if $\phi$ a formula  in the language of difference rings with parameters from $K$ and $X=\phi(K)\subseteq K^d$, then $\mathrm{int}_{\val}(X)\neq\emptyset$ if and only if $p^d_{gen}|_K$ concentrates on $X$.  Let $(K,\val,\sigma)\preccurlyeq (\monster,\val,\sigma)$, with $(\monster,\val,\sigma)$ sufficiently saturated,  and choose $a=(a_1,\ldots,a_d)$ in $\monster$ such that 
$\val(a_i)>\Gamma_{K(a_1,\ldots,a_{i-1})_{\sigma}}$ for all $i$. Then $a+b\models p^d_{gen}|_K$ for any $b\in K^d$. As for any $b\in \mathrm{int}_{\val}(X)$ we have $a+b\in \phi(\monster)$, this shows one direction.  For the converse, assume that $p^d_{gen}|_K$ concentrates on $X$. As the zero set of any non-trivial difference polynomial over $K$ is closed in the valuation topology, $p^d_{gen}|_K(\monster)$ is (non-empty and) open in the valuation topology. Thus $\mathrm{int}_{\val}(\phi(\monster))\neq\emptyset$, and so $\mathrm{int}_{\val}(X)\neq\emptyset$ by elementarity, since the interior of a definable set is definable in $\mathcal{L}_{\rf,\vg,\sigma}$.
\end{proof}

\begin{remark}\label{rm-saturatedModel}
Let $(F,\sigma)$ be an arbitrary difference field. As $\widetilde{\omega VFA}$ is the model-companion of $\omega VFA$, there is a model  $(K,\val,\sigma)\models\widetilde{\omega VFA}$ containing $(F,\sigma)$ as a trivially valued difference subfield. 

Actually, there is a more specific construction. Let $(k,\overline{\sigma})\models ACFA$ such that $(F,\sigma)\subseteq(k,\overline{\sigma})$, and let  $(\Gamma,\sigma_\Gamma)\models\widetilde{\omega OGA}$. Then the Hahn difference valued field  $\mathcal{K}=(k((t^{\Gamma})),\Gamma,k,\val_t,\sigma)$ from Example~\ref{E:Hahn}(1) is a model of $\widetilde{\omega VFA}$, containing $(F,\sigma)$ as a trivially valued difference subfield. Moreover, for an infinite cardinal $\kappa$, by \cite[Proposition~9.13]{DoHr22}, $\mathcal{K}$ is $\kappa^+$-saturated if and only if both $(k,\overline{\sigma})$ and $(\Gamma,\sigma_\Gamma)$ are $\kappa^+$-saturated.
\end{remark}

\begin{lemma}\label{lemma-lift}
    Let $(K,\Gamma,k,v,\sigma)\models \widetilde{\omega VFA}$ and let $(F,\sigma)$ be a trivially valued difference subfield of $K$. Then $(F,\sigma)$ is contained in a difference field of representatives $(\tilde{k},\sigma)$, i.e., a (trivially valued) difference subfield of $K$ such that $\res:(\tilde{k},\sigma)\cong(k,\overline{\sigma})$.
\end{lemma}

\begin{proof}
    The proof of \cite[Lemma~5.20]{DoHr22} shows this more generally for any transformally henselian model of $\omega VFA$.
\end{proof}

\begin{lemma}\label{lemma-transalg}
   Let $(F,\sigma)$ be a lift of the residue difference field in a model $(K,\Gamma_K,k_K,\sigma)$ of $\omega VFA$. Then $(F,\sigma)$ is transformally algebraically closed in $K$, namely if $a\in K$ is transformally algebraic over $F$, then $a\in F$. 
\end{lemma}
\begin{proof}
Let $a\in K\setminus F$. As $(F,\sigma)$ is a maximal trivially valued (difference) subfield of $(K,\sigma)$, the field $F(a)_{\sigma}$ is non-trivially valued, and so the value group of $F(a)_{\sigma}$ is of infinite $\Qq$-rank, since $\sigma_\Gamma$ is $\omega$-increasing. It follows that $F(a)_\sigma$ is of infinite transcendence degree over $F$, showing that   $a$ is transformally transcendental over $F$.
\end{proof}

\bigskip

\section{Local dimensions}\label{sec:loc_dim}
Let $X$ be a difference variety.   
 The goal of this section is to identify a special subvariety $X_s\subseteq X$ of strictly smaller transformal dimension, outside of which an analogue of the implicit function theorem holds.   For this purpose, we will need a local dimension for $X$, which requires a finer topology than the Zariski-Cohn topology. This topology will be provided by  
 enriching ACFA to $\widetilde{\omega VFA}$; we view our model of ACFA as the residue field
 of a model of $\widetilde{\omega VFA}$;  lifting to the valued field, we can study perturbations. We will show that if  X  is of transformal dimension d, $X$ contains a $d$
-dimensional neighborhood of each point from $X \setminus X_s$.

Recall that $\rf$ denotes the residue field sort.

\begin{theorem}\label{thm-main1}
Let $X\subseteq \Aa^n$ be a difference variety over a difference field $F$ of transformal dimension $\tdim_F(X)=d>0$. Then there exists a difference subvariety $X_s\subseteq X$ defined over $F$ with $\tdim_F(X_s)<\tdim_F(X)$, such that if $F\leq M\models \widetilde{\omega VFA}$ with $F$ trivially valued, then for each $x\in (X\setminus X_s)(\rf(M))$, in $M$ we have \[\mathrm{top.}\dim(X(\mathcal{O})\cap \res^{-1}(x))=d,\] where we write $X(\rf(M))$ when viewing $X$ as defined by parameters in $\res(F)\cong F$.
\end{theorem}
\begin{proof}
Let $\Pi$ be the set of all\footnote{up to permutations of coordinates} $\binom{n}{d}$ coordinate projections from $\Aa^n$ to $\Aa^d$. Let $M\models \widetilde{\omega VFA}$ with $F$ embedded as a trivially valued difference subfield of $M$ and $M$ being $|F|^+$-saturated. 

For a variety $Y$ over $\vf(M)$, we will write $Y(M)$ for $Y( \vf(M))$.  

Let $k:=\rf(M)$. 
Note that $k$ is naturally an $F$-algebra via the residue map. 

Fix $\pi\in\Pi$ until further notice.   Let $\bar{a}\in \rf(M)^d$ realise the generic type in $\rf^d$ over $\res(F)$. 
 For $b\in k^d$, denote by $X_b$ the fiber over $b$ for the projection $\pi$, i.e.\ $X_b(k)=\{y\in k^{n-d}: (b,y)\in X(k)\}$. Since $X$ has transformal dimension $d$ and $\bar{a}$ is generic in $\rf^d$ over $\res(F)$, $\dim_{\ACFA}(X_{\bar{a}}(k))\leq 0$.

Fix an embedding of difference fields $\eta: k=\rf(M)\hookrightarrow M$ of $k$ into the valued difference field $M$ over $F$ such that $\eta\circ\res\upharpoonright_F=id$ and $\res\circ \eta=id$. Such an $\eta$ exists by Lemma~\ref{lemma-lift}. Since $M\models \widetilde{\omega VFA}$, by Lemma~\ref{lemma-transalg}, $\eta(k)$ is transformally algebraically closed in $M$ and $\dim_{\ACFA}(X_{\eta(\bar{a})}(M))\leq 0$. 
Therefore, $X_{\eta(\bar{a})}(\eta(k))=X_{\eta(\bar{a})}(M)$, which in particular implies $X_{\eta(\bar{a})}(M)\subseteq \mathcal{O}^{n-d}$ and $\res$ is a bijection from $X_{\eta(\bar{a})}(M)=X_{\eta(\bar{a})}(M)\cap\mathcal{O}^{n-d}$ to $X_{\bar{a}}(k)$. Let $a'\in \mathcal{O}^{d}$ with $\res(a')=\bar{a}$. Since $\res(a')=\res(\eta(\bar{a}))=\bar{a}$ and $\bar{a}\in k^d$ is generic over $\res(F)$, by Corollary~\ref{cor-uniqueGeneric} $a'$ is generic over $F$ and $\tp(a'/F)=\tp(\eta(\bar{a})/F)$ in $\widetilde{\omega VFA}$. In particular, the residue map is a bijection between $X_{a'}(M)=X_{a'}(M)\cap \mathcal{O}^{n-d}$ and $X_{\bar{a}}(k)$. Therefore, for any $z\in \pi^{-1}(\bar{a})\cap X(k)$, there is $z'\in \pi^{-1}(a')\cap X(\mathcal{O})$ with $\res(z')=z$. Hence, $z'\in\res^{-1}(z)\cap X(\mathcal{O})$ and $\pi(z')=a'$, namely, $a'\in \pi(\res^{-1}(z)\cap X(\mathcal{O}))$.

At this point we let $\pi$ vary again in $\Pi$.

Let $\bar{x}=(x_1,\ldots,x_d)$ and $\psi_{\pi}(\bar{x})$ be the formula stating that:
\begin{center}
$\bar{x}\in \rf^d$ and for all $z\in\pi^{-1}(\bar{x})\cap X(\rf)$, for all $x'\in \mathcal{O}^{d}$, if $\res(x')=\bar{x}$ then $x'\in\pi(\res^{-1}(z)\cap X(\mathcal{O}))$.    
\end{center}
 Let $\Delta$ be the collection of sentences given by $\widetilde{\omega VFA}$ together with the diagram of  $F$ as an embedded trivially valued difference subfield. Our previous argument shows that the following holds for every projection $\pi\in\Pi$: \[\Delta\land (\bar{x}\in \rf^d)\land\left(\bigwedge_{P\in \res(F)[x_1,\ldots,x_d]_{\sigma}, P\text{ non-trivial}}P(\bar{x})\neq 0\right)\models \psi_{\pi}(\bar{x}),\]  since $\bar{x}$ is generic in $\rf^d$ over $\res(F)$ if and only if
 $P(\bar{x})\neq 0$ for all non-trivial polynomials $P\in\res(F)[x_1,\ldots,x_d]_{\sigma}$. By compactness, there are finitely many non-trivial difference polynomials $P_0,\ldots,P_N\in \res(F)[x_1,\ldots,x_d]_{\sigma}$ such that $\psi_{\pi}(\bar{x})$ is a consequence of  $\Delta\land (\bar{x}\in \rf^d)\land \bigwedge_{i\leq N}P_i(\bar{x})\neq 0$. Let $Z_{\pi}\subseteq \Aa^d$ be the non-empty  (as $(P_i)_{i\leq N}$ are all non-zero) open set defined by not vanishing of $P_i$ for $i\leq N$ and let $Y_{\pi}:=X\cap\pi^{-1}(Z_{\pi})$.  Let $X_s:=X\setminus\left( \bigcup_{\pi\in \Pi}Y_{\pi}\right)$, then $X_s$ is a closed subvariety of $X$ defined over $F$ (we consider $P_1,\ldots,P_N$ defined over $F$ via the isomorphism $res:F\to\res(F)$). Moreover, if $M\models \widetilde{\omega VFA}$ and $F$ is embedded into the valued difference field with trivial valuation, then for each $z\in (X\setminus X_s)(\rf(M))$, there is $\pi\in\Pi$ such that $z\in Y_{\pi}$. Let $\bar{a}:=\pi(z)$, by the definition of $Y_{\pi}$, $M\models \psi_{\pi}(\bar{a})$ and $z\in\pi^{-1}(\bar{a})\cap X(\rf(M))$. Hence, for all $a\in\mathcal{O}^d$ with $\res(a)=\bar{a}$, we have $a\in\pi(\res^{-1}(z)\cap X(\mathcal{O}))$. In particular, $\pi(X(\mathcal{O})\cap \res^{-1}(z))$ contains a valuation open subset of $\mathcal{O}^d$. 

We only need to show that $\tdim_F(X_s)<d$. Note that by definition of $X_s$, for any completion of $\ACFA$ and $\pi\in\Pi$, we have $\pi(X_s)\subseteq \Aa^d\setminus Z_{\pi}$. Since $Z_{\pi}$ is open and non-empty in $\Aa^d$, $\dim_{\ACFA}(\pi(X_s)(E))<d$ for any $F\leq E\models \ACFA$. As this holds for all $\pi:\Aa^n\to \Aa^d$ and all embeddings $F\leq E\models \ACFA$, we get $\tdim_F(X_s)<d$.\qedhere

\end{proof}

Let $X$ be a difference variety defined over some difference field $F$ and let $F\subseteq E\models\mathrm{ACFA}$. It can happen that $\dim_{\mathrm{ACFA}}(X(E))<\tdim_F(X)$. The following corollary shows that in this case, $X(E)$ has no points outside special difference subvariety $X_s$.
\begin{corollary}
Let $X$ be a difference variety defined over $F$ of transformal dimension $d$. Let $X_s$ be the subvariety of $X$ over $F$ which is of strictly smaller transformal dimension given by Theorem~\ref{thm-main1}. Suppose $F$ is embedded into $E\models \mathrm{ACFA}$. Then either $\dim_{\mathrm{ACFA}}(X(E))=d$ or $X(E)=X_s(E)$.
\end{corollary}
\begin{proof}
Let $\tilde{E}$ be a sufficiently saturated elementary extension of $E$.
Let $M\models\widetilde{\omega VFA}$ such that 
$\rf(M)=\tilde{E}$ (see Remark~\ref{rm-saturatedModel}). 
Fix a difference field lift of $\tilde{E}$ into $M$. (Such a lift exists by Lemma~\ref{lemma-lift}.) Then $F\subseteq\tilde{E}\preceq M\models \mathrm{ACFA}$. If $X(\tilde{E})\neq X_s(\tilde{E})$, then $X_s(\tilde{E})\subsetneq X(\tilde{E})$, hence there is $x\in (X\setminus X_s)(\tilde{E})$. By Theorem~\ref{thm-main1}, we have $\mathrm{top.}\dim(X(\mathcal{O})\cap \res^{-1}(x))=d$ in $M$. In particular $\mathrm{top.}\dim(X(M))\geq d$. Since the topological dimension and the transformal dimension coincide for difference varieties defined in $M$ (by Lemma~\ref{topdim=trfdim}), $\dim_{\mathrm{ACFA}}(X(M))\geq d$. By definition, we also have $\dim_{\mathrm{ACFA}}(X(M))\leq \tdim_F(X)=d$. Hence, $\dim_{\mathrm{ACFA}}(X(M))= d$. As $F\subseteq E\preceq \tilde{E}\preceq M$ in $\mathcal{L}_\sigma$, we get $\dim_{\mathrm{ACFA}}(X(E))=d$ as well. 

\end{proof}

\begin{definition}
Let $D$ be a finitely generated difference domain, and let $X\subseteq \Aa^n$ be a difference variety defined by $(P_i)_{i\leq N}$ over $D$. Given a homomorphism $f:D\to (K,\sigma)$ into a difference field $K$, we denote by $X^f$ the difference variety over $K$ defined by the polynomials $(f(P_i))_{i\leq N}$.
\end{definition}
\begin{remark}
    Note that if $Q_1,\ldots,Q_\ell$ is another set of generators of $\Sigma_D:=I(X)\cap D$ in $D[x_1,\ldots,x_n]_\sigma$, then $f(Q_1),\ldots,f(Q_\ell)$ and $(f(P_i))_{i\leq N}$ generate the same perfect difference ideal in $K[x_1,\ldots,x_n]_\sigma$. Hence the definition of $X^f$ does not depend on the choice of defining polynomials of $X$ in $D[x_1,\ldots,x_n]_\sigma$.
\end{remark}

\begin{remark}
Let $X$ be a difference variety defined over a difference domain $D$ and $f:D\to (K,\sigma)$ be a homomorphism. Then both $\tdim_K(X^f)>\tdim_D(X)$ and $\tdim_K(X^f)<\tdim_D(X)$ may occur, as the following examples show.

\begin{enumerate}
\item Let $D$ be generated over $\mathbb{Z}$ by $a$ with no relations. Let $X\subseteq\Aa^3$ be defined by $(x-1)(ay-1)=0$ and $(z-1)(ay-1)=0$. Then $\tdim_D(X)=2$. Let $f:D\to \mathbb{Q}$ map $a$ to $0$. Then $\tdim_{\mathbb{Q}}(X^f)=1$.
\item Let $f:D\to \mathbb{Q}$ be as above and $Y\subseteq \Aa^3$ be defined by $(\sigma(y)-a)(y-z)=0$ and $(\sigma(y)-y)(y-x)=0$. Then $\tdim_D(Y)=1$, but $Y^f$ contains a component $\Aa^1\times\{0\}\times\Aa^1$, hence is of transformal dimension 2 over $\mathbb{Q}$.
\end{enumerate}
\end{remark}

The next lemma shows that if we invert finitely many elements in $D$, then we can guarantee that the transformal dimension of $X^f$ does not go up.
\begin{lemma}\label{lem:local}
 For a difference domain $D$ with fraction field $F$, for any difference variety $X$ over $F$ of transformal dimension $d$, there is a difference domain $D'$ with $D\subseteq D'\subseteq F$ and $D'$ finitely generated over $D$, such that $\tdim_K(X^f)\leq \tdim_D(X)$ for all homomorphisms $f: D'\to (K,\sigma)$.
\end{lemma}

\begin{proof}
    Let $\Pi_{d+1}$ be the collection of projections $\Aa^n\to\Aa^{d+1}$. Since $X\subseteq \Aa^n$ has transformal dimension $d$, there are non-trivial difference polynomials $\{P_{\pi}(x_1,\ldots,x_{d+1}):\pi\in\Pi_{d+1}\}$ over $F$ such that if $F\leq E\models \mathrm{ACFA}$, then $\pi(X(E))\subseteq Y_{\pi}(E)$ where $Y_{\pi}$ is the difference variety in $\Aa^{d+1}$ defined by the equation $P_\pi(x_1,\ldots,x_{d+1})=0$. In other words, $\qftp_{\mathcal{L}_\sigma}(F)\cup\mathrm{ACFA}\models \bigvee_{\pi\in\Pi_{d+1}}(\pi(X)\subseteq Y_{\pi}\land Y_{\pi}\neq \Aa^{d+1})$. By compactness, some finite $\Delta\subseteq \qftp_{\mathcal{L}_\sigma}(F)$ together with $\mathrm{ACFA}$ are sufficient on the left hand side. Since $F$ is a difference field, by adding inverses to elements that are prescribed to be non-zero by $\Delta$, there is a finite subset $A\subseteq F$ such that the positive quantifier-free $\mathcal{L}_\sigma$-type of $A$ implies $\Delta$. 
    
    Let $D':=D[A]_\sigma$. If $f:D'\to K$ is a homomorphism to a difference field, then $f(A)\subseteq K$ satisfies $\Delta$. Hence $X^f$ has transformal dimension at most $d$.
\end{proof}

The next theorem is a uniform version of Theorem~\ref{thm-main1}.

\begin{theorem}\label{MainTheorem}
Let $X\subseteq\Aa^n$ be a difference variety of transformal dimension $d$ defined over a finitely generated difference domain $D=\mathbb{Z}[d_1,\ldots,d_m]_\sigma/I$. Suppose $\tdim_K(X^f)\leq d$ for every homomorphism $f$ from $D$ to a difference field $(K,\sigma)$.

Then there exists a quantifier-free $\mathcal{L}_\sigma$-formula $\varphi(x,y)$ with $x=(x_1,\ldots,x_n)$ and $y=(y_1,\ldots,y_m)$ such that for every homomorphism $f:D\to (K,\sigma)$, $\varphi(x,f(d_1),\ldots,f(d_m))$ defines a difference subvariety $X_s^{f}$ of $X^f$ over $K$, such that 
\begin{itemize}
    \item 
    $\tdim_K(X_s^{f})<d$, and
    \item 
    if $M\models \widetilde{\omega VFA}$ with $K\subseteq \mathcal{O}_M$, then for each $x\in (X^f\setminus X_s^f)(\rf(M))$, in $M$ we have \[\mathrm{top.}\dim(X^f(\mathcal{O}_M)\cap \res^{-1}(x))=d.\] 
\end{itemize}

\end{theorem}

\begin{proof}
Since $D$ is finitely generated, it is a Ritt ring. By Noetherian induction, it suffices to show that there is $d'\in D\setminus\{0\}$ and a difference subvariety $X_s$ of $X$ defined over $D$ such that the conclusion of the theorem holds for all homomorphisms $f:D\to (K,\sigma)$ with $f(d')\neq 0$. Indeed, for the induction, one considers $X^f$ for $f:D\to \Frac(D/\mathfrak{p})$ where $\mathfrak{p}$ is a minimal reflexive prime difference ideal containing $d'$. Note that $X^f$ is a difference variety defined over $D/\mathfrak{p}=\mathbb{Z}[d_1,\ldots,d_m]/I'$ of transformal dimension $\leq d$ by assumption.

Let $X_s$ be as given in Theorem~\ref{thm-main1}. Then $X_s$ is defined over the fraction difference field $F$, hence is also defined over $D$ since $D$ is Ritt.

By Lemma~\ref{lem:local}, there is a finitely generated difference domain $D''\supseteq D$ contained in $F$, such that $\tdim_K((X_s)^f)\leq \tdim_D(X_s)$ for all homomorphisms $f:D''\to K$.

Let $\Theta_F$ be $\widetilde{\omega VFA}\cup\qftp_{\mathcal{L}_\sigma}(F)\cup\{v(a)\geq0,a\in F\}$. By Theorem~\ref{thm-main1}, 
\begin{align*}
    \Theta_F\models & \forall x\in (X\setminus X_s)(\rf),  \\
    &\pi(X(\mathcal{O})\cap \res^{-1}(x)) \text{ contains an open ball for some $\pi:\Aa^{n}\to\Aa^d$.}
\end{align*}

By compactness, a finite subset $\Psi\subseteq \Theta_F$ implies the above. Hence, there is a finite set $A'\subseteq F$, closed under inverse, such that the positive $\mathcal{L}_\sigma$-quantifier-free type of $A'$ and that $A'$ has non-negative valuation in a model of $\widetilde{\omega VFA}$ imply $\Psi=\Psi(A')$. Let $D'$ be the difference domain generated by $A'$ over $D''$. If $f:D'\to K$ is a homomorphism and $K$ is embedded into the valuation ring of some model of $\widetilde{\omega VFA}$, so is $f(A')$. Since $f$ is a homomorphism, $f(A')$ satisfies the positive $\mathcal{L}_\sigma$-quantifier-free type of $A'$, thus $\Psi(f(A'))$ is satisfied.

Finally note that since $D'\subseteq F$ is finitely generated and contains $D$, $D'\subseteq D[1/d']\sigma$ for some $d'\in D\setminus\{0\}$.
\end{proof}

\section{Equidimensionality of Frobenius reductions}\label{sec:equi_dim}
Recall that we denote by $K_q$ the difference field $(\mathbb{F}_p^{alg},\frob_q)$.
In this paper, when we refer to algebraic varieties, we mean algebraic sets defined by some set of polynomials in affine space, which are neither required to be irreducible nor reduced. Since our primary interest is their rational points, this causes no harm.
\begin{definition}
Let $X$ be a difference variety over $K_q$. Define $M_q(X)$ to be the algebraic variety over $\mathbb{F}_p^{alg}$ defined by the algebraic polynomials obtained by replacing any term $\sigma(x)$ by $x^q$ in the defining difference polynomials of $X$.

Let $X$ be a difference variety over some finitely generated difference domain $D$. Given a homomorphism $\eta:D\to K_q$ of difference domains, we call the algebraic variety $M_q(X^\eta)$ \emph{a Frobenius reduction} of $X$.
\end{definition}

\begin{example}

For any $q=p^n$, the mod $p$ map is the only morphism of difference rings $\eta:\mathbb{Z}\to K_q$. 
    
On the other hand, let $D=\mathbb{Z}[i]$, then more interesting Frobenius reductions appear. For example, one can define a map $\eta: D\to \Ff_3^{alg}$ by fixing $\xi$ a square root of $-1$ in $\Ff_3^{alg}$ and define $\eta(i)=\xi$, and we have automorphisms swapping the two square roots of $-1$ (in $D$ and $\Ff_3^{alg}$).

\end{example}

\begin{example}\label{eg:dim-fail}
    If $X$ has transformal dimension $d>0$ over $D$, it is not the case that there exists $D'\subseteq \mathrm{Frac}(D)$ such that $M_q(X^\eta)$ has algebraic dimension $d$ for all homomorphisms $\eta:D'\to K_q$, it is even not the case for $q$ sufficiently large. For example, let $X$ be defined by $P_1(x_1,x_2):=\sigma(x_1)-x_1$ and $P_2(x_1,x_2):=x_1^2+x_1+1$ over $\mathbb{F}_2$. Namely, $x_1\neq 1$ is a cube root of unity and is in the fixed field. Since $3\mid (2^{2n}-1)$ and $3\nmid (2^{2n+1}-1)$ for all $n$, we see that one can find such $x_1$ in $(\mathbb{F}_2^{alg},\frob_{2^{2n}})$ but they don't exist in $(\mathbb{F}_2^{alg},\frob_{2^{2n+1}})$. Hence, $M_{2^{2n}}(X)$ has dimension 1 while $M_{2^{2n+1}}(X)$ has dimension $-\infty$ for all $n$.
\end{example}

In this section, we will show that the dimension of Frobenius reductions of $X$ can be controlled modulo $X_s$ in the following sense: If $X$ has transformal dimension $d>0$, the Frobenius reduction $M_q((X\setminus X_s)^\eta)$ is either empty or breaks into irreducible subvarieties of algebraic dimension $d$ for all $q>N_X$ for some constant $N_X$. Here $X_s$ is given as in Theorem~\ref{MainTheorem}. Intuitively, $X_s$ is the set of points where an analogue of the implicit function theorem fails, which could be treated as a ``singular locus" to some extent. The content of the results in this section is that away from $X_s$, the Frobenius reductions preserve necessarily geometric information.

Moreover, we will prove a uniform version of this result, which says that if $(X_b)_b$ is a definable family of difference varieties of transformal dimension $d$ over $D_b$, then $(X_b)_s$ also varies in a definable family and there is a common bound for $N_{X_b}$. To do this, we need to introduce the notion of complexity of a difference variety.

\begin{definition}
    Let $X$ be a difference variety given by difference polynomials $(P_i)_{i\leq m}$ with coefficients in a difference domain $D$. We define the complexity of $X$ as $\sum_{i\leq m}|P_i|$, where $|P_i|$ is the length of $P_i$ treated as an $\mathcal{L}_{\sigma,D}$-formula where $\mathcal{L}_{\sigma,D}$ is the language of difference rings augmented with constant symbols in $D$, namely $\mathcal{L}_{\sigma,D}:=\{+,-,0,1,\sigma,(d)_{d\in D}\}$. Note that if $\eta:D\to E$ is a morphism of difference domains, then $X^\eta$ has complexity bounded by the complexity of $X$.
\end{definition}

\begin{remark}
    \begin{itemize}
        \item 
        The complexity of $X$ depends on the choice of defining polynomials. From now on all difference varieties come with a fixed choice of defining polynomials. 
        \item 
        Let $E$ be a difference field. Given any $N,n>0$, the family of difference varieties of $\Aa^n_E$ defined over $E$ of complexity bounded by $N$ is a definable family in $E$, since the number and the length of defining difference polynomials are all bounded by $N$. Moreover, it is trivial but worth pointing out that every variety has finite complexity.
    \end{itemize}
\end{remark}

 Note that in what follows, we only require the fact that $\dim(M_q(X))\leq \tdim_{K_q}(X)$ for $q$ sufficiently large from the following lemma. This inequality allows for a more elementary proof.

\begin{lemma}\label{L:Trasnf-AlgDim}
    For all $e>0$ there is a constant $B(e)$ such that for all $q=p^n>B(e)$, whenever $X$ is a difference variety of complexity $\leq e$ defined over $K_q$, then $\dim(M_q(X))=\tdim_{K_q}(X)$.
\end{lemma}

    \begin{proof}
Recall that $\mathcal{L}_q$ is $(\Ff_q(t)^{alg},\frob_q)$ with valuation $\val_t$ the $t$-adic valuation of $\Ff_q(t)$ extended to the algebraic closure. Note that $\mathcal{L}_q$ is an elementary extension of $K_q$ as difference fields. It thus suffices to prove the result for $\mathcal{L}_q$ instead of $K_q$. Assume for contradiction that for any $i\in \mathbb{N}$ there is $q_i>B$ and a difference variety $X_{b_i}$ of complexity $\leq e$ defined over $\mathcal{L}_{q_i}$ such that $\tdim_{\mathcal{L}_{q_i}}(X_{b_i})=d$ and $\dim(M_{q_i}(X_{b_i}))\neq d$.

By Lemma~\ref{ultratrd}, letting 
$L:=\prod_{i\rightarrow\mathcal{U}}\mathcal{L}_{q_i}$ and $X=\prod_{i\to\mathcal{U}}X_{b_i}\subseteq \Aa^n_{(L,\sigma)}$, we get 
$\tdim_L(X)=d$. Assume $\mathcal{U}$ is non-principal. 
By Fact~\ref{omegaVFA}, $L\models\widetilde{\omega VFA}$, so Lemma~\ref{topdim=trfdim} yields 
$\mathrm{top.}\dim(X)=d$. On the other hand, in the local models 
$\mathcal{L}_{q_i}$ we get $\mathrm{top.}\dim(X_{b_i}(\mathcal{L}_{q_i}))=\mathrm{top.}\dim(M_{q_i}(X_{b_i}))=\dim(M_{q_i}(X_{b_i}))\neq d$, which is a contradiction.
\end{proof}

In the proof of Theorem~\ref{thm-equiDim} below, we will use the following fact which is presumably well-known. We omit the proof of this fact, which is elementary.

\begin{fact}\label{F:Lift}
Let $K$ be a valued field, $E\subseteq K$ a lift of the residue field $k$ of $K$, and let $Y$ be an algebraic variety defined over $E\cong k$. Then $\res(Y(\mathcal{O}_K))\subseteq Y(k)$.
\end{fact}

\begin{theorem}\label{thm-equiDim}
    For all $e>0$, there is $C=C(e)>0$, such that the following holds.

Let $X\subseteq\Aa^n$ be a difference variety defined over some finitely generated difference domain $D$ of complexity $\leq e$ with $\tdim_D(X)=d>0$. Then there is a difference subvariety $X_s$ over $D$ of complexity $\leq C$ with $\tdim_D(X_s)<d$ and a finitely generated difference domain $D'$ with $D\subseteq D'\subseteq \Frac(D)$ and such that for all $q>C$ and for all Frobenius reductions $\eta:D'\to K_q$, $M_q((X_s)^\eta)$ has algebraic dimension strictly smaller than $d$ and $M_q((X\setminus X_s)^\eta)$ is either empty or equidimensional of algebraic dimension $d$, namely, every irreducible component is of dimension $d$. 

\end{theorem}

\begin{remark}
In particular, suppose $K=(K,\sigma)$ is a difference field and $\{X_a,a\in A\subseteq K^m\}$ is a family of difference varieties of transformal dimension $d$ defined by a fixed set of difference polynomials with varying parameters $a\in A$, then $\{(X_a)_s:a\in A\}$ is also uniformly definable by another fixed set of difference polynomials with each $(X_a)_s$ defined by parameters in the difference field generated by $a$. 
\end{remark}

\begin{proof}
  Assume to the contrary, there is some natural number $e$ such that for any $i>0$, there is a difference variety $X_i$ of transformal dimension $d$ of complexity $\leq e$ defined over some finitely generated $D_i\subseteq F_i:=\Frac(D_i)$, with no $(X_i)_s\subseteq X_i$ of complexity $\leq i$ defined over $D_i$ and $D_i\subseteq D'_i\subseteq F_i$ with $D_i'$ finitely generated satisfying $M_q((X_i\setminus (X_i)_s)^\eta)$ is equidimensional of algebraic dimension $d$ and $M_q(((X_i)_s)^\eta)$ has algebraic dimension $<d$ for every Frobenius reduction $\eta:D_i'\to K_q$ with $q>i$.

    Let $\mathcal{U}$ be a non-principal ultrafilter and define $D:=\prod_{i\to\mathcal{U}}D_i$ and $F:=\prod_{i\to\mathcal{U}}F_i$. Since the complexity of each $X_i$ is bounded by $e$, $(X_i)_i$ is a uniformly definable family of difference varieties of complexity $\leq e$. Therefore, $X:=\prod_{i\to\mathcal{U}}X_i$ also has complexity $\leq e$. Now, by Lemma \ref{ultratrd}, $X$ has transformal dimension $d$ over $F$. Since $\Frac(D)=F$, $X$ has transformal dimension $d$ over $D$. Let $D_0\subseteq D$ be a finitely generated difference subdomain such that $X$ is defined over $D_0$ and $\tdim_{D_0}(X)=d$. By Lemma~\ref{lem:local}, replacing $D_0$ with $D_0=D_0[b_0]_\sigma$ for some finite $b_0\in \Frac(D_0)$, we may assume that $\tdim_L(X^f)\leq d$ for every homomorphism $f$ from $D_0$ to a difference field $(L,\sigma)$. Let $X_s$ and $D'=D_0[b]_\sigma\subseteq \Frac(D_0)$ be given as in Theorem~\ref{MainTheorem}, with $b$ a finite tuple. Write $X_s=\prod_{i\to\mathcal{U}}(X_i)_s$, then $(X_i)_s$ is defined over $D_i$ of transformal dimension $<d$ for ultrafilter-many $i$. Write $b=(b_i)_{i\to\mathcal{U}}$ with $b_i\in F_i$. Let $D'_i:=D_i[b_i]_\sigma$. Take $C$ to be the complexity of $X_s$. By assumption, for all $i>C$, there is $\eta_i:D'_i\to K_{q_i}$ with $q_i>i$ such that either $M_{q_i}((X_i\setminus (X_i)_s)^{\eta_i})$ is not equidimensional of algebraic dimension $d$ or $M_{q_i}(((X_i)_s)^{\eta_i})$ has algebraic dimension $\geq d$. We get a homomorphism $\eta:\prod_{i\to\mathcal{U}}D'_i\to \prod_{i\to\mathcal{U}}K_{q_i}=:E$ which restricts to a homomorphism $\eta: D'\to E$. By Theorem~\ref{MainTheorem}, $\tdim_E(X^\eta)\leq d$ and $\tdim_E((X_s)^\eta)\leq \tdim_{D'}(X_s)<d$. 
 Let $B\subseteq E$ be the fraction field of the difference domain generated by $\eta(D')$.
    
    Note that $X^\eta=\prod_{i\to\mathcal{U}}(X_i)^{\eta_i}$ and $(X_s)^\eta=\prod_{i\to\mathcal{U}}((X_i)_s)^{\eta_i}$.
    
    By Lemma~\ref{L:Trasnf-AlgDim} it then follows that for ultrafilter-many $i$, the algebraic dimensions of $M_{q_i}((X_i)^{\eta_i})$ and $M_{q_i}(((X_i)_s)^{\eta_i})$ are bounded above by $d$ and $d-1$ respectively. We conclude that $M_{q_i}((X_i\setminus (X_i)_s)^{\eta_i})$ is not equidimensional for ultrafilter-many $i$.

Therefore, we may assume some irreducible component $S_i$ of $M_{q_i}((X_i\setminus (X_i)_s)^{\eta_i})$ has algebraic dimension $<d$ and the algebraic dimension of $M_{q_i}(((X_i)_s)^{\eta_i})$ is $<d$ for all $i\in\mathbb{N}$. Take $x_i\in S_i(\mathbb{F}_{p_i}^{alg})$ that is not in any other irreducible components of $M_{q_i}((X_i\setminus (X_i)_s)^{\eta_i})$ and $x=(x_i)_{i\to\mathcal{U}}$. 

Setting $M:=\prod_{i\to\mathcal{U}}M_i:=\prod_{i\to\mathcal{U}}((\Ff_{p_i}(t)^{alg},\Ff_{p_i}^{alg},\frob_{q_i})=\prod_{i\to\mathcal{U}}\mathcal{L}_{q_i}$, Fact~\ref{omegaVFA}(4) yields $M\models \widetilde{\omega VFA}$. Then $E=\rf(M)$ is the residue difference field of $M$ and $B$ embeds into $K:=\prod_{i\to\mathcal{U}}\Ff_{p_i}(t)^{alg}$ as a trivially valued difference subfield. 
Now $x\in (X\setminus X_s)^\eta(E)$. By Theorem~\ref{MainTheorem}, for some $\pi:\Aa^n\to \Aa^d$, \[\pi(X^\eta(\mathcal{O})\cap \res^{-1}(x))\text{ contains a non-empty valuation open ball in }K^d.\] Hence for ultrafilter-many $i\in\mathcal{U}$, $\pi((X_i)^{\eta_i}(\mathcal{O}_i)\cap \res^{-1}(x_i))$ contains a non-empty valuation open ball in $(K_i)^d$, where $K_i:=\Ff_{p_i}(t)^{alg}$, since containing a non-empty valuation open ball is a definable condition. By assumption, $x_i\in M_{q_i}((X_i\setminus (X_i)_s)^{\eta_i})$ lies in an irreducible component $S_i=Y_i\setminus M_{q_i}(((X_i)_s)^{\eta_i})$ of dimension $<d$, where $Y_i$ denotes the Zariski closure of $S_i$ in $M_{q_i}((X_i)^{\eta_i})$, which is easily seen to be an irreducible component of $M_{q_i}((X_i)^{\eta_i})$, of algebraic dimension $<d$. 
We claim $(X_i)^{\eta_i}(\mathcal{O}_i)\cap \res^{-1}(x_i)\subseteq Y_i(\mathcal{O}_i)$, since if $y\in (X_i)^{\eta_i}(\mathcal{O}_i)$ with $\res(y)=x_i$ and $y$ is in some other irreducible component $Z\neq Y_i$ of $M_{q_i}((X_i)^{\eta_i})$, then $x_i=\res(y)\in Z$ by Fact~\ref{F:Lift}, contradicting our choice of $x_i$.
Now we have both the facts that $Y_i$ has algebraic dimension $<d$ and that $\pi((X_i)^{\eta_i}(\mathcal{O}_i)\cap \res^{-1}(x_i))\subseteq \pi(Y_i(\mathcal{O}_i))$ contains a non-empty valuation open ball in $(K_i)^d$,  a contradiction to the fact that algebraic dimension and topological dimension agree for definable sets in any model of ACVF.
\end{proof}

\begin{definition}
\begin{itemize}
    \item Let $X$ be a difference variety over $D$ with $\tdim_D(X)=d$ and let $U$ be a Zariski-Cohn open subset of $X$. We say $U$ is \emph{Frobenius equidimensional} if for all Frobenius reduction $\eta: D\to K_q$, $M_q(U^\eta)$ is either empty or equidimensional of algebraic dimension $d$.
    \item Given a difference variety $X$ defined over $D$, we call $X_s$ and $D'$ in Theorem~\ref{thm-equiDim} \emph{the special subvariety of $X$} and \emph{the special domain of reduction of $X$} respectively. 
\end{itemize}
\end{definition}
\begin{remark}\label{rem: nb irreducible components}
    Theorem~\ref{thm-equiDim} tells us that $M_q((X\setminus X_s)^\eta)$ is equidimensional for all Frobenius reductions $\eta$. However, it does not indicate how many irreducible components $M_q((X\setminus X_s)^\eta)$ can have. Note that the degree of $M_q((X\setminus X_s)^\eta)$ is bounded by $c'q^c$ for some constants $c$ and $c'$, thus the number of irreducible components is also bounded by it. 

    In general the number of irreducible components of the Frobenius reduction $M_q(V^\eta)$ of a difference variety $V$ depends on $q$. However, we speculate that when the generic of $V$ is orthogonal to varieties of transformal dimension 0, we have a better bound.
    \begin{question}
    Let $M\models \mbox{ACFA}$ and $B$ be an algebraically closed difference subfield and $a\in M^n$. Suppose there is no $b\in\acl(aB)\setminus B$ with $\trfdeg(b/B)=0$. Is it true that there is a locally closed difference variety $V$ with $a\in V(M)$ defined over some difference subdomain $D\subseteq B$  and a constant $C$, such that $M_q(V^\eta)$ has at most $C$ irreducible components for all Frobenius reduction $\eta:D\to K_q$?
    \end{question}
\end{remark}

\section{Twisting reduction and twist-rational morphisms}\label{sec:twist_red}

This section gives the final preparation for the proof of our main theorem. The basic setting is: suppose $X\subseteq\Aa^n_D$ is a difference variety with $\tdim_D(X)=d>0$ and $\eta:D\to K_q$ be a homomorphism and we want to estimate the size of the set of rational points of $M_q(X^\eta)$ in some finite difference field $\mathbb{F}_{p^t}$. By the previous sections, we know that $M_q((X\setminus X_s)^\eta)$ is either empty or breaks into absolutely irreducible components $(Y_j)_j$ of dimension $d$. If we could show that $Y_j$ contains a smooth point in $\mathbb{F}_{p^t}$, then a standard argument shows that $Y_j$ is defined over $\mathbb{F}_{p^t}$. See Example~\ref{eg:dim-fail}. Then by the Lang-Weil estimate, we could conclude that $Y_j$ contains approximately $p^{td}$-many points. The aim of this section is to divide $X$ into pieces so that each piece is already ``smooth'' in terms of Jacobian criterion. 

More precisely, the main result of this section is that we could cut $X$ into finitely-many locally closed subvarieties $X_i$ such that each $X_i$ is homeomorphic to some ``nice'' locally closed subvariety $Y_i$. Here ``nice'' means that $Y_i\subseteq\Aa^{d_i+1}$ is of transformal dimension $d_i$ and $Y_i$ is smooth in the sense that there is a difference polynomial $Q(x_1,\ldots,x_{d_i+1})$ such that $Y_i$ vanishes along $Q$ but $\frac{\partial Q}{\partial x_1}$ is non-zero everywhere on $Y_i$, where partial derivative is defined as in Section~\ref{sec:pre1}.

However, there are some difficulties in formulating the clean statement above mathematically. Suppose $X\subseteq \Aa^2_K$ is defined over some inversive difference field $K=(K,\sigma)$ by $P(x_1,x_2):=a\sigma(x_1)-b\sigma(x_2)$. Then the partial derivatives of $P$ always vanish. However, $P$ generates the same perfect difference ideal as $(\sigma^{-1}P)(x_1,x_2):=\sigma^{-1}(a)x_1-\sigma^{-1}(b)x_2$. Particularly, they define the same zero set. 

There is yet another difficulty. If $\mathrm{char}(K)=p>0$ and $X$ is defined by $Q(x_1,x_2):=x_1^p\sigma(x_2)$. Then all the partial derivatives of $Q$ vanish as well. The way to fix this is to define an automorphism $\tau$ on $K$ such that $\sigma=\tau\cdot\frob_p$ and write $Q$ as a $\tau$-difference polynomial $Q_\tau(x_1,x_2):=x_1^p\tau(x_2^p)$, and then change $Q_\tau(x_1,x_2)$ to $(\frob_p^{-1}P_\tau)(x_1,x_2):=x_1\tau(x_2)$. Notationally, we will also use $a\mapsto a^{1/p}$ to denote $\frob_p^{-1}$. Now if we view $X$ as a $(K,\tau)$ difference variety defined by $(\frob_p^{-1}P_\tau)(x_1,x_2)$, then generically it does not vanish on partial derivatives. This procedure is called the \emph{twisting reduction}. The algebro-geometric analogue of this is the relative Frobenius construction, see~\cite[Section 1.2]{ez}. In conclusion, it only makes sense to work with perfect inversive difference fields $(K,\sigma)$, and by homeomorphisms, we could mean homeomorphisms of the Zariski-Cohn topology where we replace the automorphism $\sigma$ with some element in the subgroup of $\mathrm{Aut}(K)$ of the form $\sigma\circ \frob^\ell$ for some $\ell\in\Zz$.

We will first introduce the twisting reduction of difference polynomials and then establish a generic homeomorphism if two irreducible difference varieties have the same function fields up to taking perfect inversive closure. We then use the primitive element theorem to embed generically an irreducible difference variety of dimension $d$ to $\Aa^{d+1}$. Finally, we combine all these steps to prove the main result which in addition works uniformly.

\begin{definition}\label{Def-Twist}
Let $(K,\sigma)$ be a perfect inversive difference field of characteristic $p\geq 0$ and $P(x_1,\ldots,x_n)\in K[x_1,\ldots,x_n]_{\sigma}$ be a non-constant difference polynomial. 

\begin{enumerate}
\item  Let $\ell\in\Zz$, the $\ell$'th twist of $\sigma$ is the automorphism $\tau$ of $K$ such that $\tau\circ(\frob_p)^\ell=\sigma$.
\item An \emph{$\ell$'th twisting reduction} of $P$ is a difference polynomial $Q\in K[x_1,\ldots,x_n]_{\tau}$, where $\tau$ is an $\ell$'th twist of $\sigma$ for some $\ell$, obtained by the following two steps:
\begin{itemize}
    \item
    Let $m\geq 0$ be a natural number such that $P\in K[x_1^{\sigma^m},x_2^{\sigma^m},\ldots,x_n^{\sigma^m}]_{\sigma}$ (we set $x^{\sigma^0}:=x$) and let $\tilde{P}(x_1,\ldots,x_n):=\sigma^{-m}(P)$. Then $\tilde{P}\in K[x_1,\ldots,x_n]_{\sigma}$ by our choice of $m$ and as $K$ is inversive.
    \item
    Write $\tilde{P}:=\sum_{j\leq J}a_j\prod_{1\leq i\leq n}x_i^{\mu_{i,j}}$ where $a_j\in K\setminus\{0\}$ and $\mu_{i,j}\in\Nn[\sigma]$ such that if $j\neq j'$, then $(\mu_{1,j},\ldots,\mu_{n,j})\neq (\mu_{1,j'},\ldots,\mu_{n,j'})$. Let $S:=\{s\in\Nn^{>0}:\text{ exists }\mu_{i,j}=s+\sum_{1\leq t\leq t_{i,j}}n_t\sigma^{t}\}.$  Now let $\ell\geq 0$ be a natural number such that $p^\ell$ divides all numbers in $S$ (if $p=0$, set $\ell=0$). We define a morphism of semi-rings $f:\Nn[\sigma]\to\Nn[\tau]$ by setting $f(\sigma):=p^\ell\tau$. Clearly, $f$ is an injective map. Define $\tilde{Q}(x_1,\ldots,x_n)\in K[x_1,\ldots,x_n]_{\tau}$ where $\tau$ is the $\ell$'s twist of $\sigma$ as $\tilde{Q}:=\sum_{j\leq J}a_j\prod_{1\leq i\leq n}x_i^{f(\mu_{i,j})}$. Since $p^\ell$ divides all the elements in $S$, necessarily $p^\ell$ divides all coefficients appearing in $f(\mu_{i,j})\in\Nn[\tau]$ for all $i,j$. Now we set 
    \[Q(x_1,\ldots,x_n):=\sum_{j\leq J}a_j^{1/p^\ell}\prod_{1\leq i\leq n}x_i^{f(\mu_{i,j})/p^\ell}\in K[x_1,\ldots,x_n]_{\tau}.\]
   \end{itemize} 
   \item We say that a twisting reduction $Q$ of $P$ is \emph{full} if in the first step $m$ is chosen maximal such that $P\in K[x_1^{\sigma^m},x_2^{\sigma^m},\ldots,x_n^{\sigma^m}]_{\sigma}$, and in the second step $\ell$ is chosen maximal such that $p^\ell$ divides all the elements in $S$. (Note that by maximality of $m$, necessarily $S\neq \emptyset$, and so there is indeed a maximal such $\ell$.)
   \end{enumerate}
   \end{definition}

\begin{lemma}\label{lem:twist-reduction}
If $(K,\sigma)$ is a perfect inversive difference field and $P\in K[x_1,\ldots,x_n]_{\sigma}$ is non-constant, the full twisting reduction $Q$ of $P$ is well-defined, and $Q$ has a non-zero partial derivative $\frac{\partial Q}{\partial x_{i_0}}$ for some $1\leq i_0\leq n$.
    \end{lemma}
    
    \begin{proof}
    Existence and uniqueness of the full twisting reduction $Q\in K[x_1,\ldots,x_n]_{\tau}$ are clear. Now suppose $Q=\sum_{j\leq J}a_j\prod_{1\leq i\leq n}x_i^{\mu_{i,j}}$ with $a_j\in K\setminus\{0\}$ and $\mu_{i,j}\in\Nn[\tau]$ such that if $j\neq j'$, then $(\mu_{1,j},\ldots,\mu_{n,j})\neq (\mu_{1,j'},\ldots,\mu_{n,j'})$. By construction of $Q$, there is some $(i_0,j_0)$ such that $\mu_{i_0,j_0}=n_t\tau^t+\cdots+n_1\tau+n_0$ with $n_0>0$ and not divisible by $p$. Now treat $Q$ as an algebraic polynomial in $x_{i_0}$ with coefficients in $K[x_1,...,x_{i_0-1},\tau(x_{i_0}),x_{i_0+1},...,x_n]_\tau$. So $Q(x_1,...,x_n)=c_0+c_1x_{i_0}+...+c_{n_0}x^{n_0}_{i_0}+...+c_mx^m_{i_0}$ where $\frac{\partial c_j}{\partial x_{i_0}}=0$ for each $j$ (as $c_j\in K[x_1,...,\tau(x_{i_0}),...,x_n]_\tau$) and $c_{n_0}$ is non-trivial (since $\mu_{i_0,j_0}$ has constant term $n_0$ and $(\mu_{1,j},\ldots,\mu_{n,j})\neq (\mu_{1,j'},\ldots,\mu_{n,j'})$ for $j\neq j'$, hence cancellation is not possible). It then follows that $\frac{\partial Q}{\partial x_{i_0}}$ is non-trivial as it contains the non-trivial term $n_0c_{n_0}x^{n_0-1}_{i_0}$.
    \end{proof}

   \begin{remark}
   Suppose $(K,\sigma)\leq (L,\sigma)$ are perfect and inversive and $P\in K[x_1,\ldots,x_n]_{\sigma}$ is non-constant. Let $Q\in K[x_1,\ldots,x_n]_{\tau}$ be a twisting reduction of $P$ for some twist $\tau$ of $\sigma$. Then $P$ and $Q$ have the same zero sets in $L$.
   \end{remark}

In the following, we will give an analogue of generic universal homeomorphism between irreducible varieties, which in particular induce homeomorphisms in the Zariski-Cohn topology over perfect inversive difference fields extensions generically.

Recall that the \emph{perfect inversive closure} of a difference field $L$ is by definition the smallest perfect inversive difference field 
containing $L$. Also recall that  $K[x_1,\ldots,x_n]^{\pm}_\sigma$ denotes the perfect inversive closure of the difference polynomial ring $K[x_1,\ldots,x_n]_\sigma$.

\begin{definition}
Let $X$ and $Y$ be two irreducible difference varieties over a perfect inversive difference field $K$. We say $X$ and $Y$ are \emph{twist-birational} over $K$ if the perfect inversive closures of their function fields $K(X)_\sigma$ and $K(Y)_\sigma$ are isomorphic over $K$. 
\end{definition}

\begin{definition}
 Let $(K,\sigma)$ be a difference field.
A function of the form $(P_1/Q_1, \ldots,P_m/Q_m)$ is called a \emph{twist-rational function over $K$} if $P_i,Q_i\in K[x_1,\ldots,x_n]^{\pm}_\sigma$ such that $Q_i\neq0$ for all $i$. It is said to be \emph{defined} at $a$ if $Q_i(a)\neq0$ for all $i$.   
\end{definition}

\begin{remark}
Suppose $L=K(a_1,\ldots,a_n)_\sigma$ and $F=K(b_1,\ldots,b_m)_\sigma$ are two finitely generated difference fields over a perfect inversive difference subfield $K$. Denote by $\tilde{L}$ and $\tilde{F}$ the perfect inversive closures of $L$ and $F$ respectively. Suppose further that $h:\tilde{L}\cong\tilde{F}$ is a difference field isomorphism over $K$. Then there are twist-rational functions $(P_1/Q_1, \ldots,P_m/Q_m)$ defined at $\bar{a}:=(a_1,\ldots,a_n)$ and $(R_1/S_1, \ldots,R_n/S_n)$ defined at $\bar{b}=(b_1,\ldots,b_m)$ such that $h(a_j)=R_j(\overline{b})/S_j(\overline{b})$ and $h^{-1}(b_i)=P_i(\overline{a})/Q_i(\overline{a})$ for all $i,j$.
\end{remark}

The following is a general topological fact we need. We leave the elementary proof for the reader.

\begin{fact}\label{fact-homeo}
Let $X,Y$ be irreducible topological spaces. Suppose $f:X\dashrightarrow Y$ and $g:Y\dashrightarrow X$ are continuous functions defined on a non-empty open set $X'\subseteq X$ and a non-empty open $Y'\subseteq Y$ respectively. Suppose there are non-empty open sets $X_0\subseteq X$ and $Y_0\subseteq Y$ such that $g\circ f=\mathrm{id}$ on $X_0$ and $f\circ g=\mathrm{id}$ on $Y_0$. Then $f$ and $g$ are homeomorphisms between the non-empty open sets $X_0':=f^{-1}(Y_0)\cap X_0$ and  $Y_0':=g^{-1}(X_0)\cap Y_0$.
\end{fact}

\begin{comment}
\begin{lemma}\label{lem-homeo}
Let $Y$ be an irreducible topological space. Suppose $f:X\dashrightarrow Y$ and $g:Y\dashrightarrow X$ are continuous functions defined on an open set $X'\subseteq X$ and an open $Y'\subseteq Y$ respectively. Suppose there are non-empty open sets $X_0\subseteq X$ and $Y_0\subseteq Y$ such that $g\circ f=id$ on $X_0$ and $f\circ g=id$ on $Y_0$. Then $f$ and $g$ are homeomorphisms between $X_0$ and the open set $Y_0':=g^{-1}(X_0):=\{y\in Y':g(y)\in X_0\}$.
\end{lemma}
\begin{proof}
It is clear that $f\circ g$ is defined on $Y_0'$ and $g(Y_0')=X_0$. We only need to show that $f\circ g=id$ on $Y_0'$. Consider the set $\{y\in Y_0':f\circ g(y)=y\}$, it is a closed set in $Y_0'$ for $f\circ g$ is continuous on $Y_0'$. \textcolor{red}{emmm, why the diagonal is closed in $Y\times Y$?} And it contains the open set $Y_0\cap Y_0'$. Therefore, there is closed set $\tilde{Y}\subseteq Y$ such that $\{y\in Y_0':f\circ g(y)=y\}=\tilde{Y}\cap Y_0'$. Now $\tilde{Y}$ is closed in $Y$ and contains a non-empty open set $Y_0\cap Y_0'$, by irreducibility of $Y$, we conclude $\tilde{Y}=Y$ and \[\{y\in Y_0':f\circ g(y)=y\}=\tilde{Y}\cap Y_0'=Y_0'.\]\qedhere
\end{proof}
\end{comment}

\begin{lemma}\label{lem-openhomeo}
Suppose $X\subseteq \Aa^n_{(K,\sigma)}$ and $Y\subseteq \Aa^m_{(K,\sigma)}$ are irreducible difference varieties defined over a perfect inversive difference field $K$. 
Suppose $X$ and $Y$ are twist-birational over $K$. Then there are non-empty open subvarieties $X_0\subseteq X$ and $Y_0\subseteq Y$ and twist-rational functions $f$ and $g$ defined over $K$ such that for all perfect inversive difference field extension $E$ of $K$, $f$ induces a bijection from $X_0(E)$ to $Y_0(E)$ with $g$ being the inverse. Indeed, they are homeomorphisms with respect to the Zariski-Cohn topology over $E$.
\end{lemma}
\begin{proof}
Let $K[x_1,\ldots,x_n]_{\sigma}/I(X)$ and $K[y_1,\ldots,y_m]_\sigma/I(Y)$ be the coordinate rings of $X$ and $Y$ respectively.  Let $L=K(X)_\sigma$ and $F=K(Y)_\sigma$ be their function fields and $\tilde{L}$, $\tilde{F}$ be their perfect inversive closures. Set $\alpha_i:=x_i\mod I(X)$, $\beta_j:=y_j\mod I(Y)$ and $\bar{\alpha}=(\alpha_1,\ldots,\alpha_n)$, $\Bar{\beta}=(\beta_1,\ldots,\beta_m)$.
 We may assume $\tilde{L}=\tilde{F}$ and hence there are $P_i,Q_i\in K[x_1,\ldots,x_n]^{\pm}_\sigma$ for $1\leq i\leq m$ and $R_j,S_j\in K[y_1,\ldots,y_m]^{\pm}_\sigma$ for $1\leq j\leq n$ with $\beta_i=P_i(\overline{\alpha})/Q_i(\overline{\alpha})$ and $\alpha_j=R_j(\overline{\beta})/S_j(\overline{\beta})$ for all $1\leq i\leq m$ and $1\leq j\leq n$.

Suppose $E$ is a perfect inversive difference field extension of $K$. Let $\tilde{E}\models ACFA$ be a sufficiently saturated ambient difference field containing $E$. Define the partial map $f:X(\tilde{E})\to Y(\tilde{E})$ as follows. Given $(a_1,\ldots,a_n)\in X(\tilde{E})$, let $b_i:=(P_i/Q_i)(a_1,\ldots,a_n)$ for $1\leq i\leq m$ and set $f(a_1,\ldots,a_n)=(b_1,\ldots.b_m)$. Then $f$ is defined on the open subset $X'(\tilde{E})$  of $X(\tilde{E})$ given by $\bigwedge_{1\leq i\leq m}Q_i(x_1,\ldots,x_n)\neq 0$. And if $(a_1,\ldots,a_n)\in X'(\tilde{E})$, then $f(a_1,\ldots,a_n)\in Y(\tilde{E})$. Clearly for any perfect inversive subfield $L$ containing $K$, $f$ restricts to a partial map from $X(L)$ to $Y(L)$. Moreover, $f:X'(\tilde{E})\to Y(\tilde{E})$ is a continuous map with respect to the Zariski-Cohn topology on $X'(\tilde{E})$ and $Y(\tilde{E})$ over $K$, since any closed set in $Y(\tilde{E})$ is given by some set of equations $\bigwedge_{i\leq N}T_i(y_1,\ldots,y_m)=0$. It has preimage in $X'(\tilde{E})$ given by the equations $\bigwedge_{i\leq N}T_i(P_1/Q_1,\ldots,P_m/Q_m)(x_1,\ldots,x_n)=0$ (we may write $T_i(P_1/Q_1,\ldots,P_m/Q_m)(x_1,\ldots,x_n)=0$ as an equation $T(x_1,\ldots,x_n)=0$ by chasing the denominators and applying $\sigma$ and $\frob_p$ repeatedly, and it will have the same set of solutions in $X'(\tilde{E})$.) Similarly, we define the partial map $g:Y(\tilde{E})\to X(\tilde{E})$ as: given $(b_1,\ldots,b_m)\in Y(\tilde{E})\to X(\tilde{E})$, let $a_j:=(R_j/S_j)(b_1,\ldots,b_m)$ for $1\leq j\leq n$ and define $g(b_1,\ldots,b_m)=(a_1,\ldots,a_n)$. Then $g$ is defined on an open set $Y'(\tilde{E})$ and is continuous and restricts to a partial map from $Y(L)$ to $X(L)$ for any perfect inversive subfield $L$ containing $K$. 

Let $(a_1,\ldots,a_m)\in X(\tilde{E})$ be a generic point over $K$. \footnote{Depending on the choice of $\tilde{E}$, namely the completion of ACFA over $K$, there may not be such a generic point, or even no point at all.} Then $f(a_1,\ldots,a_m)$ is a generic point in $Y(\bar{E})$ over $K$ (as $K(a_1,\ldots,a_m)_\sigma=L=K(X)_\sigma$.) By definition, $g\circ f(a_1,\ldots,a_m)=(a_1,\ldots,a_m)$. As $g\circ f$ is defined on the open set $X_0':=X'\cap X''$ where $X''$ is given by $\bigwedge_{i\leq n} S_i(P_1/Q_1,\ldots,P_m/Q_m)(x_1,\ldots,x_n)\neq 0$. Therefore, $g\circ f$ is continuous on $X_0'$ and the set $\{x\in X_0:g\circ f(x)=x\}$ is a closed set in $X_0'$ containing the generic $(a_1,\ldots,a_n)$. By irreducibility of $X$, we conclude $g\circ f$ is the identity map on the open set $X_0'$. Similarly, $f\circ g$ is the identity map on a non-empty open set $Y_0'$. By Fact~\ref{fact-homeo} $f:X_0(\tilde{E})\to Y_0(\tilde{E})$ is a homeomorphism where $X_0(\tilde{E}), Y_0(\tilde{E})$ are the open sets defined by $X_0'(\tilde{E})\cap f^{-1}(Y_0'(\tilde{E}))$ and $Y_0'(\tilde{E})\cap f^{-1}(X_0'(\tilde{E}))$ respectively. Note that $X_0$ and $Y_0$ only depend on $f$ and $g$ and do not depend on $E$ or $\tilde{E}$. Now use the fact that $X_0(E)=X_0(\tilde{E})\cap E^n$ and $Y_0(E)=Y_0(\tilde{E})\cap E^m$ for a perfect inversive subfield $E$ containing $K$. It is easy to see that $f$ restricts to a homeomorphism from $X_0(E)$ to $Y_0(E)$ in the Zariski-Cohn topology restricted to rational points in $E$. Since $f$ does not depend on $E$, $f$ is indeed a homeomorphism for the Zariski-Cohn topology over any perfect inversive difference field extension of $K$. 
\end{proof}

\begin{remark}\label{rem:homeo}
    Note that from the proof above we have the following: If $X$ and $Y$ are locally closed varieties and $f,g$ twist-rational functions over $F$ such that $f$ induces a bijection between $X(E)$ and $Y(E)$ with $g$ being the inverse, for all perfect inversive difference field extension $E$ of $F$, then $f,g$ induce homeomorphisms in the Zariski-Cohn topology over any such $E$. 
    We call such maps \emph{twist-birational homeomorphisms}.
\end{remark}

The last part of this section will be devoted to proving the main result that we mentioned in the beginning. We will need the primitive element theorem for difference fields in the proof. Recall the definitions of limit degree, reduced limit degree and non-periodic difference fields in section \ref{sec:pre1}.

\begin{fact}[\mbox{\cite[Chapter 7, Theorem III]{cohn}}]\label{fact: primititve element} Let $F=(F,\sigma)$ be a non-periodic difference field and $E$ be a finitely generated transformally algebraic difference field extension of $F$. If $\rld(E/F)_\sigma=\ld(E/F)_\sigma$, then there exists an element $\gamma\in E$ and a natural number $t$ such that $\sigma^t(a)\in F(\gamma)_\sigma$ for all $a\in E$.
\end{fact}

\begin{lemma}\label{lem-rld}
Suppose $(K,\sigma)$ is a difference field of characteristic $p$. Let $L=K(\bar{a})_{\sigma}$ be a finitely generated difference field extension of $K$ which is transformally algebraic over $K$. Then there is $\tau:=\sigma\circ (\frob_p)^{t}$ for some $t\in\Nn$ and $n\in\Nn$ such that $\ld(F/K)_{\tau}=\rld(F/K)_{\tau}$ where $F:=K(\bar{a},\sigma(\bar{a}),\ldots,\sigma^{n}(\bar{a}))_{\tau}\leq L$.
\end{lemma}

\begin{proof}

For $p=0$, the statement is trivially true. For $p>0$,
take a natural number $n_0$ large enough so that for $L_0:=K(\bar{a},\sigma(\bar{a}),\ldots,\sigma^{n_0}(\bar{a}))$ and $L_1:=K(\bar{a},\sigma(\bar{a}),\ldots,\sigma^{n_0+1}(\bar{a}))$ we have \[\ld(L/K)_\sigma=p^tm=[L_{1}:L_{0}]\] and $\rld(L/K)_{\sigma}=m=[L_{1}:L_{0}]_s$ for some $m$ and $t$. Let $\bar{b}=(\bar{a},\sigma(\bar{a}),\ldots,\sigma^{n_0}(\bar{a}))$ and $L_n:=K(\bar{b},\sigma(\bar{b}),\ldots,\sigma^{n}(\bar{b}))=K(\bar{a},\sigma(\bar{a}),\ldots,\sigma^{n+n_0}(\bar{a}))$. Then $[L_{n+1}:L_{n}]=\ld(L/K)_\sigma=p^tm$ and $[L_{n+1}:L_{n}]_s=\rld(L/K)_\sigma=m$. We use $L'_{n+1}$ to denote the separable part of $L_{n+1}$ over $L_{n}$.
Let $\tau:=\sigma\circ(\frob_p)^t$. We may regard $L$ as a $\tau$-difference field naturally (but possibly not finitely generated as a $\tau$-difference field over $K$). Let $F=K(\bar{b})_{\tau}\leq L$ and $F_n:=K(\bar{b},\tau(\bar{b}),\ldots,\tau^{n}(\bar{b}))$.

Note that $F_0=L_0=K(\bar{b})$ and $F_n\leq L'_n$ for all $n\geq 1$, since by definition of $L_n'$, for any $x\in L_n$, $x^{p^t}\in L'_n$ and in particular $\tau^i(\bar{b})=(\sigma^i(\bar{b}))^{p^{it}}\in L'_n$ for all $1\leq i\leq n$. Therefore, $[F_1:F_0]\leq [L'_1:F_0=L_0]\leq m$ and hence $\rld(F/K)_{\tau}\leq \ld(F/K)_{\tau}\leq m$.

It remains to show that $\rld(F/K)_\tau=m$.  Note that by our construction, it is not hard to check that $L_n$ is a purely inseparable extension of $F_n$.

By the multiplicativity of separable degrees in towers, we have that \[m=[L_n:L_{n-1}]_s=[L_n:F_{n-1}]_s=[L_n:F_n]_s[F_n:F_{n-1}]_s=[F_n:F_{n-1}]_s.\] It follows that $\rld(F/K)_\tau=\rld(L/K)_\sigma=m$. 
\begin{comment}
  The following is the diagram of the above computation. We compute the degrees in two ways, and the extensions from left to right are purely inseparable.
\begin{center}
\begin{tikzcd}
               &                               & L_n                                            \\
               & L'_n \arrow[ru]              &                                                \\
F_n \arrow[ru] &                               & L_{n-1} \arrow[lu]                             \\
               & F_{n-1} \arrow[lu] \arrow[ru] & 
\end{tikzcd}
\end{center}  
\end{comment}
\end{proof}

\begin{definition}
Let $(K,\sigma)$ be a perfect, inversive $\sigma$-difference field and $X\subseteq\Aa^n_{(K,\sigma)}$ be a $\sigma$-difference variety given by difference polynomials $\{P_i(x_1,\ldots,x_n)=0:i\leq N\}$. Let $\tau$ be the $\ell$'th twist of $\sigma$ for some $\ell\in \Zz$. We may regard $X$ as a $\tau$-variety $X_\tau\subseteq \Aa^n_{(K,\tau)}$ in the following way: we replace $\sigma(t)$ with $\tau(t^{p^{\ell}})$ in the polynomials $P_i$ for any occurrence of $\sigma(t)$ for any $\mathcal{L}_{ring}$-term $t$ and then applying $p^{th}$ power repeatedly until there is no occurrence of $t^{p^z}$ for negative $z$. In this way, we get a $\tau$-difference polynomial $P_i'$. Now we take $X_{\tau}$ to be defined by the perfect difference ideal generated by $\{P_i'(x_1,\ldots,x_n)=0:i\leq N\}$.
\end{definition}
\begin{remark}
Let $K(X)_\sigma=K(a_1,\ldots,a_n)_\sigma$ be the function field of $X$ and $L$ be the perfect inversive closure of $K(X)_\sigma$. Let $F:=K(a_1,\ldots,a_n)_{\tau}$ be the $\tau$-difference subfield of $(L,\tau)$. Then $F$ is the function field of $X_{\tau}$ over $(K,\tau)$. 
\end{remark}

\begin{theorem}\label{thm-pi-twist-uniform}

Let $X\subseteq\Aa^n_D$ be a difference variety with $\tdim_D(X)=d>0$ where $D=\mathbb{Z}[b]_\sigma$ is a finitely generated difference domain of characteristic $p\geq 0$ for some $b=(b_1,\ldots,b_m)$. Then there is
\begin{itemize}
    \item 
    a sequence $d=d_0\geq d_1\geq\cdots\geq d_N>0$;
    \item 
   locally closed varieties $(X_i)_{i\leq N}$ with each $X_i$ open in the closed variety $X\setminus \bigcup_{j<i}X_j$ over $D$ with $\tdim_D(X_i)=d_i$, and $\tdim_D(X\setminus\bigcup_{i\leq N}X_i)=0$;
    \item 
    $D'=D[1/c]_\sigma$ with $c\in D\setminus\{0\}$;
    \item 
    twists $(\tau_i)_{i\leq N}$ of $\sigma$;
    \item 
    quantifier-free $\mathcal{L}_{\sigma,\sigma^{-1},\frob_p^{-1},\tau_0,\ldots,\tau_N}$-formulas $(\psi_{i}(t^i;y))_{i\leq N}$, $(\theta_i(x,t^i,y))_{i\leq N}$ with $x=(x_1,\ldots,x_n)$, $y=(y_1,\ldots,y_m)$ and $t^i=(t_1,\ldots,t_{d_i+1})$;
    \item 
    and $\tau_i$-polynomials $(Q_i(t^i)_{i\leq N}$,
\end{itemize}
such that the following holds:

For every homomorphism $\eta: D'\to (K,\sigma)$ with $K$ perfect inversive,
\begin{enumerate}

    \item 
    $\psi_{i}(t^i;\eta(b))$ defines a locally closed $\tau_i$-subvariety $Y^\eta_i$ in $\Aa^{d_i+1}_{(K,\tau_i)}$ for each $i$;
    \item 
    $\theta_i(x,t^i;\eta(b))$ defines a twist-birational homeomorphism between the $\sigma$-variety $X^\eta_i(K)$ and the $\tau_i$-variety $Y^\eta_i(K)$;
   
    \item 
    $Q_i(a_1,\ldots,a_{d_i+1},\eta(b))=0$,
    $(\partial Q_i(t^i,\eta(b))/\partial t_1)(a_1,\ldots,a_{d_i+1})\neq 0$ for all $(a_1,\ldots,a_{d_i+1})\in Y^\eta_i(K)$. In particular, $Q_i(t^i,\eta(b))$ is not the zero polynomial in $K[t_1,\ldots,t_{d_i+1}]_{\tau_i}$,
\end{enumerate}
where we consider $(K,\sigma)$ with the natural $(K,\sigma,\sigma^{-1},\frob_p^{-1},\tau_0,\ldots,\tau_N)$-expansion.
\end{theorem}
\begin{proof}
It suffices to show this result for $\eta$ injective. Indeed, let $F$ denote the perfect inversive closure of $\Frac(D)$ with the canonical inclusion $\iota:D\to F$. 
For every perfect inversive difference field $E$ and injective $\eta:D\to E$, it factors through $\iota$. If the result holds for all injective $\eta$, the positive quantifier-free type of $\Frac(D)$ with the axioms of perfect inversive difference fields implies (1) to (3). Note that they are all first-order expressible, for example to express (2), it is enough to say $\theta_i$ defines a twist-birational bijection, and homeomorphism follows from Remark~\ref{rem:homeo}.
By compactness, a finite subset $\Sigma$ of the positive quantifier-free type of $\Frac(D)$ is sufficient. Such $\Sigma$ can always be preserved after inverting the product of finitely many elements in $D$.

By Noetherian induction, it is enough to find an open set $X_0\subseteq X$ and $Y_0$ satisfying the conclusion of the theorem.
By looking at each component, we may further assume $X$ is irreducible in the Zariski-Cohn topology over $\Frac(D)$ and $\tdim_D(X)=d>0$. 

Let $F[x_1,\ldots,x_n]_{\sigma}/I(X)$ be the coordinate ring of $X$ and $L=F(X)_\sigma$ be the function field and note that $\tdim_F(X)=d$. Up to relabelling, we may assume $x_1,\ldots,x_{d}$ is a transformal transcendence basis of $L$ over $F$. Let $F':=F(x_1,\ldots,x_d)_{\sigma}\leq L$ and $\bar{a}=(x_{d+1},\ldots,x_n)$. Then $L=F'(\bar{a})_\sigma$ is finitely generated and transformally algebraic over $F'$. Note $F'$ is non-periodic since $d>0$. By Lemma~\ref{lem-rld}, there is some $\tau=\sigma\circ(\frob_p)^t$ for some $t\in \Nn$ and $n_0\in \Nn$ such that $\ld(C/F')=\rld(C/F')$ where $C=F'(\bar{a},\sigma(\bar{a}),\ldots,\sigma^{n_0}(\bar{a}))_\tau$. Then by Fact~\ref{fact: primititve element}, there is $\gamma\in C$ and $s\in\Nn$ such that $\tau^s(C)\subseteq F'(\gamma)_{\tau}$. Particularly, for $M:=F(x_1,\ldots,x_d,\gamma)_{\tau}$ and $E:=F(x_1,\ldots,x_d,\bar{a})_{\tau}$ with $\tilde{M}$ and $\tilde{E}$ denoting the perfect inversive closure of $M$ and $E$ respectively, we have $\tilde{M}=\tilde{E}$.

Let $P\in F[t_1,\ldots,t_{d+1}]_\tau\setminus F[\tau(t_1),\ldots,\tau(t_{d+1})]_\tau$ such that  $P(x_1,\ldots,x_d,\gamma)=0$. We choose $P$ of minimal total difference degree with these properties. 
By Lemma~\ref{lem:twist-reduction}, if $Q(t_1,\ldots,t_{d+1})\in F[t_1,\ldots,t_{d+1}]_{\tau'}$ denotes the full twist reduction of $P$, where $\tau'$ is some twist of $\tau$, in particular a twist of $\sigma$, there is $i_0$ such that $\frac{\partial Q}{\partial t_{i_0}}$ is non-trivial. 

\begin{claim-star}
We have $\frac{\partial Q}{\partial t_{i_0}}(x_1,\ldots,x_d,\gamma)\neq0$.
\end{claim-star}

\begin{proof}

To prove the claim, assume that $\tau=\tau'\circ(\frob_p)^r$, where $r\in\mathbb{N}$, and that with the notation from Definition~\ref{Def-Twist} we have $P(\overline{t})=\sum_{j\in J}a_j\prod_{1\leq i\leq d+1}t_i^{\mu_{i,j}}$.

Then by construction, 
\[Q(\overline{t})=\sum_{j\in J}a^{1/p^r}_j\prod_i t^{f(\mu_{i,j})/p^r}_i,
\]
for $f:\mathbb{N}[\tau]\to \mathbb{N}[\tau']$, $\tau\mapsto p^r\tau'$.

In the following, we will show that $\frac{\partial Q}{\partial t_{i_0}}$ is the $r$'th twist reduction of a (non-zero) polynomial $P_1(\bar{t})\in K[\bar{t}]_\tau$ of strictly smaller total difference degree than $P$. As $P_1(x_1,\ldots,x_d,\gamma)=0$ if and only if $\frac{\partial Q}{\partial t_{i_0}}(x_1,\ldots,x_d,\gamma)=0$, we arrive at the conclusion of the claim.

Note that any monomial in $P$ is of the form $P_0(\bar{t})=a t_{i_0}^{p^rn_0}\tilde{P}_0(\bar{t})$ for some $n_0\geq 0$ and $a\in F\setminus\{0\}$ with $\tilde{P}_0$ a polynomial in $K[t_1,\ldots,t_{i_0-1},\tau(t_{i_0}),t_{i_0+1},\ldots,t_{d+1}]_\tau$, with $r$'th twist reduction $Q_0$, then $\frac{\partial Q_0}{\partial t_{i_0}}$ is either $0$ or the $r$'th twist reduction of $bt_{i_0}^{p^r(n_0-1)}\tilde{P}_0(\bar{t})$ for some $b\in F\setminus\{0\}$ and $n_0>0$. Note in particular, every monomial in $Q$ is the $r$'th twist reduction of some monomial in $P$. It follows that $\frac{\partial Q}{\partial t_{i_0}}$ is the $r$'th twist reduction of some unique $P_1(\bar{t})$ and the total difference degree of $P_1$ is strictly smaller than the one of $P$.

\end{proof}

After relabelling variables, we may assume $t_{i_0}=t_{1}$. Note that $\tau'=\sigma\circ(\frob_p)^z$ for some $z\in\Zz$. Let $E':=F(x_1,\ldots,x_d,\bar{a})_{\tau'}\leq \tilde{E}$ and $M':=F(x_1,\ldots,x_d,\gamma)_{\tau'}\leq \tilde{M}$, then $E'$ and $M'$ have perfect inversive closures $\tilde{E}$ and $\tilde{M}$ and $(\tilde{E},\tau')=(\tilde{M},\tau')$, since $(\tilde{E},\tau)=(\tilde{M},\tau)$.

Let $Y$ be the $\tau'$-variety defined by the vanishing $\tau'$-polynomials of $(x_1,\ldots,x_d,\gamma)$ in $M'$ over $F$. Then $Q\in I(Y)$ by construction and $\tdim_F(Y)=d$.  

Let $X_{\tau'}\subseteq \Aa^n_{(F,\tau')}$ be the $\tau'$-variety defined by $I((x_1,\ldots,x_{d},\bar{a})/F)$.
Then $X_{\tau'}$ and $Y$ are twist-birational over $F$, and so by Lemma~\ref{lem-openhomeo}, there are non-empty open subsets $X'_0\subseteq X_{\tau'}$ and $Y'_0\subseteq Y$ and twist-rational functions $f,g$ over $F$ such that $f:X'_0(H)\to Y'_0(H)$ is a homeomorphism with inverse $g$ for any perfect inversive $\tau'$-difference field extension $H$ of $(F,{\tau'})$.

Let $U_0$ be defined by $\frac{\partial Q}{\partial t_{1}}\neq 0$ and $Y_0:=Y'_0\cap U_0$. Then $Y_0$ is open in the closed subvariety $Y\subseteq \Aa^{d+1}_{(F,\tau')}$ of dimension $d$ and $\frac{\partial Q}{\partial t_{1}}$ is everywhere non-zero on $Y_0$. Let $X_0$ be the open set in $X$ over $F$ such that $(X_0)_{\tau'}=f^{-1}(Y_0'\cap U_0)\subseteq X_0'$. Note that since $F$ is the perfect inversive closure of $\Frac(D)$, we may chase denominators and compose with $\frob_p$ and $\sigma$ to make $X_0$ defined over $D$. Note that both $Y_0$ and $f$ can be defined using quantifier-free $\mathcal{L}_{\sigma,\sigma^{-1},\frob_p^{-1},\tau_0,\ldots,\tau_N}$-formulas over $D$. Thus $X_0$, $Y_0$, $f$ and $Q$ satisfy the conclusion.
\end{proof}

\begin{remark}\label{rem: uniformstratification}
    By the same argument as in the proof of Theorem~\ref{thm-equiDim}, the constant $N$, the complexity of $(X_i)_{i\leq N}$, and the complexity of all formulas $\psi_i,\theta_i$ and twists $\tau_i$ only depend on the complexity of $X$ and do not depend on $D$.
\end{remark}

\section{Counting in finite difference fields}\label{sec.Count}

In this section, we prove our main theorem concerning an estimate of the number of rational points of difference varieties uniformly over finite difference fields. To this end, we need to use an improved version of the Lang-Weil estimate from~\cite{cafure}. We begin by recalling some ad-hoc notions of degree for affine varieties.
\begin{definition}
Let $V\subseteq \mathbb{A}^n$ be an affine irreducible variety over $K$. Define the \emph{degree} of $V$, $\widetilde{\deg}(V)$, to be the maximal number of points lying in the intersection of $V=V(K^{alg})$ and an affine linear subspace $L$ of $\mathbb{A}^n$ of codimension $\adim(V)$ for which $|V\cap L|<\infty$ holds.

More generally, if $V$ is an affine $K$-variety, let $V=\bigcup_{i\leq N}W_i$ be the decomposition of irreducible $K$-components. We define the degree of $V$ as:
\[\widetilde{\deg}(V):=\sum_{i\leq N}\widetilde{\deg}(W_i).\]
\end{definition}
\begin{remark}\label{rmk: degree}
     Let $V\subseteq \mathbb{A}^n$ be a hypersurface defined by $f$. Then it follows immediately from the definition that $\widetilde{\deg}(V)\leq \deg(f)$. (See also \cite[Proposition I.7.6]{Hartshorne}.)
\end{remark}

\begin{fact}[\mbox{\cite[Section 2]{cafure}}]\label{bezoutineq}
If $V$ and $W$ are subvarieties of $\mathbb{A}^n$ over $K$, then
\[\widetilde{\deg}(V\cap W)\leq\widetilde{\deg}(V)\widetilde{\deg}(W).\]
\end{fact}

We are now ready to state the improved Lang-Weil estimate that we will use later.
\begin{fact}[\mbox{\cite[Theorem 7.1]{cafure}}]\label{lem-counting}
Let $V$ be an absolutely irreducible affine algebraic variety. Assume that $V$ is defined over $\mathbb{F}_q$, of dimension $r>0$ and $\widetilde{\deg}(V)=:\ell$. If $q>2(r+1)\ell^2$, then the following estimate holds:
$$||V(\mathbb{F}_q)|-q^r|\leq(\ell-1)(\ell-2)q^{r-\frac{1}{2}}+5\ell^{\frac{13}{3}}q^{r-1}.$$
\end{fact}

\begin{lemma}\label{lem-degreeBound}
Suppose $X\subseteq \Aa^n$ is a difference variety defined over a finitely generated difference domain $D$. Then there is a constant $c$ depending only on the complexity of $X$, such that $\widetilde{deg}(M_q(X^\eta))\leq q^c$ for any Frobenius reduction $\eta:D\to K_q$. In particular, the degree of any irreducible component of $M_q(X^\eta)$ and the number of irreducible components of $M_q(X^\eta)$ are bounded by $q^c$.
\end{lemma}
\begin{proof}
Let $I$ be the defining difference ideal of $X$, generated by $f_1,...,f_m$. Then $M_q(X^\eta)$ is defined by $\eta_q(f_1),...,\eta_q(f_m)$, where $\eta_q(f)$ is the (algebraic) polynomial obtained from $\eta(f)$ where one replaces all occurrences of $\sigma$ with $\frob_q$. Note that there is a bound on the degree of the polynomials $\eta_q(f_1),...,\eta_q(f_m)$ in terms of the order of the $f_i$'s and of $q$. Thus by Remark~\ref{rmk: degree} and Fact~\ref{bezoutineq}, \[\widetilde{\deg}(M_q(X^\eta))\leq\prod_{1\leq i\leq m}\deg\eta_q(f_i)\leq q^c\] for some $c$ only depending on the complexity of $\{f_1,\ldots,f_m\}$. The in particular part follows by the definition of degree for non-irreducible varieties.
\end{proof}

The following fact is an easy consequence of the Jacobian criterion.

\begin{fact}\label{lem-JacobianCriterion}
Suppose $X\subseteq\Aa^n$ is an algebraic variety of dimension $d>0$ defined by a set of polynomials $\{P_i(x_1,\ldots,x_n):i\leq m\}\subseteq K[x_1,\ldots,x_n]$ for some $m\geq n-d$. Suppose $b\in X(K)$ is such that the Jacobian $(\frac{\partial P_i}{\partial x_j}(b))_{i,j}$ has rank at least $n-d$, and $b$ is contained in an irreducible component $Y$ over $K$ of $X$ of algebraic dimension $d$, then $b$ is a smooth point of $X$.
\end{fact}

\begin{theorem}\label{thm-finiteryCount}
Let $X\subseteq\Aa^n_D$ be a difference variety of transformal dimension $d>0$ defined over a difference domain $D=\mathbb{\Zz}[b]_\sigma$ for some $b=(b_1,\ldots,b_m)$, such that for any homomorphism $\eta:D\to (K,\sigma)$, $\tdim(X^\eta)\leq d$. Then there is a quantifier-free $\mathcal{L}_\sigma$-formula $\psi_{\xi}(x,y)$ with $x=(x_1,\ldots,x_n), y=(y_1,\cdots,y_m)$ and constants $c,C>0$ (depending only on the complexity of $X$) such that for all $q>C$, for any Frobenius reduction $\eta:D\to K_q$, $\psi_{\xi}(x;\eta(b))$ defines a difference subvariety $X_{\xi}^\eta$ of $X^\eta$ of transformal dimension $<d$ such that for all $t\in\mathbb{N}^{>0}$ with $\eta(D)\subseteq \mathbb{F}_{p^t}$.
    \begin{enumerate}
    \item 
    Either there is a point $a\in (X^\eta\setminus X^\eta_{\xi})(\mathbb{F}_{p^t},\frob_q)$, and we have \[|X^\eta(\mathbb{F}_{p^t},\frob_q)|\geq p^{dt}-q^cp^{t(d-1/2)};\]
    \item 
    Or $X^\eta(\mathbb{F}_{p^t},\frob_q)\subseteq X^\eta_{\xi}(\mathbb{F}_{p^t},\frob_q)$, and $|X^\eta(\mathbb{F}_{p^t},\frob_q)|\leq q^cp^{t(d-1)}$.
\end{enumerate}
\end{theorem}

\begin{remark}\label{rem-upperBound}
Note that $X^\eta(\mathbb{F}_{p^t},\frob_q)=M_q(X^\eta)(\mathbb{F}_{p^t})$ for any $X$ and Frobenius reduction $\eta$ to $K_q$. For $q$ large enough, we have a trivial upper bound for the number of rational points of $X^\eta$ in $(\mathbb{F}_{p^t},\frob_q)$ without the assumption of existence of the point $a$, namely  \[|X^\eta(\mathbb{F}_{p^t},\frob_q)|\leq q^{c_0}p^{dt},\] for some constant $c_0$ not depending on $q$ and $t$. Indeed, by Lemma~\ref{lem-degreeBound}, the degree of $M_q(X^\eta)$ is bounded by $q^{c_0}$ for some $c_0$. Also the dimension of $M_q(X^\eta)$ is bounded by $d$ for $q$ large enough, hence $|M_q(X^\eta)(\mathbb{F}_{p^t})|\leq q^{c_0}p^{dt}$ by \cite[Corollary 2.2]{Lachaud}.
And this upper bound cannot be essentially improved, since one could take $X\subseteq\Aa^2$ defined by $\sigma(x)=x$, then \[|X^\eta(\mathbb{F}_{p^t},\frob_q)|= qp^{t},\] for any Frobenius reduction to $K_q$ and $\Ff_{p^t}\supseteq\Ff_q$. 
\end{remark}

\begin{proof}
We will first prove the statement for some $D'=D[1/f]_\sigma\supseteq D$ with $f=f(b)\in D\setminus\{0\}$. 
By Noetherian induction, the result will follow.  

Let $D''=D[1/g]_\sigma$, $d=d_0\geq d_1\geq\ldots \geq d_N>0$ and $(X_i)_{i\leq N}$ with each $X_i$ open subvariety of the closed variety $X\setminus \bigcup_{j<i}X_j$ of transformal dimension $d_i$ be given as in Theorem~\ref{thm-pi-twist-uniform}. Let $k\leq N$ be the largest with $d_k=d$. Let $X':=X\setminus \bigcup_{i\leq k}X_i$. Then $X'$ is a closed subvariety of $X$ of transformal dimension $<d$. For each $i=\{0,\ldots,k\}$, let $X_{i,s}$ and $D_{i,s}=D[1/f_i]_\sigma$ be the special subvariety and special domain of $X\setminus \bigcup_{j<i}X_j$ as given by Theorem~\ref{thm-equiDim}, note that $X_i\setminus X_{i,s}$ is still Frobenius equidimensional, since if $V$ is an absolutely irreducible component of $X_i\setminus X_{i,s}$, then $V$ has dimension $d$ and $X_i\cap V$ is either empty or of dimension $d$ (as $X_i$ is open). By construction, $\tdim_D(X_{i,s})<d$ for all $i\leq k$. Let $X_\xi:=X'\cup\bigcup_{i\leq k}X_{i,s}$. Clearly, $\tdim_D(X_\xi)<d$. Note that the complexity of $X_\xi$ depends only on the complexity of $X$ by Remark~\ref{rem: uniformstratification}. 

Let $D'=D[1/f]_\sigma\supseteq D[1/(g\prod_i f_i)]_\sigma$, such that $\tdim(X_\xi^\eta)<d$ for any homomorphism $\eta:D'\to K$.
There is some $C$ such that if $q>C$, then $M_q(X^\eta_{\xi})$ has algebraic dimension $<d$ for any Frobenius reduction $\eta:D'\to K_q$. Now assume that $\eta(D')\subseteq \mathbb{F}_{p^t}$. Note that if $M_q(X^\eta)(\mathbb{F}_{p^t})\subseteq M_q(X^\eta_{\xi})(\mathbb{F}_{p^t})$, then $|M_q(X^\eta)(\mathbb{F}_{p^t})|\leq q^cp^{t(d-1)}$ for some constant $c$, since $M_q(X^\eta_{\xi})$ has algebraic dimension $<d$. Moreover, $c$ only depends on the complexity of $X$, since it is determined by the complexity of $X_\xi$.

If there exists $a\in M_q(X^\eta\setminus X^\eta_{\xi})(\mathbb{F}_{p^t})$, then $a\in M_q(X_i^\eta\setminus X^\eta_{\xi})$ for some $i\leq k$. Let $F_{t,q}:=(\mathbb{F}_{p^t},\frob_q)$. Then $\eta:D'\to F_{t,q}$ by the assumption that $\eta(D')\subseteq \mathbb{F}_{p^t}$.

By Theorem~\ref{thm-pi-twist-uniform}, there exist a twist $\tau_i$ of $\sigma$ and locally closed $\tau_i$-subvariety $Y_i^\eta$ of $\Aa^{d+1}_{F_{t, q'}}$ where $F_{t, q'}:=(\mathbb{F}_{p^t},\frob_{q'})$ and $\frob_{q'}$ is the $\tau_i$ twist of $\frob_q$, such that $X_i^\eta(F_{t, q})$ and $Y_i^\eta(F_{t,q'})$ are twist-birational homeomorphic by some function $\beta$. 
Note that $\beta$ can be extended to a twist-birational homeomorphism from $X_i^\eta(K_{q})$ to $Y_i^\eta(K_{q'})$ (by the fact that the same formula defines these two homeomorphisms in Theorem~\ref{thm-pi-twist-uniform}). 
By the same proposition, there is also a non-zero polynomial $Q\in F_{t,q'}[t_1,\ldots,t_{d+1}]_{\tau_i}$ such that $Q$ vanishes on $Y_i^\eta$ and $(\partial Q/\partial t_1)(\beta(a))\neq 0$.

Since both $\frob_q$ and $\frob_{q'}$ are definable in $\mathcal{L}_{ring}$,
and $X_i^\eta(K_q)=M_q(X_i^\eta)(\mathbb{F}_p^{alg})$, $Y_i^\eta(K_{q'})=M_{q'}(Y_i^\eta)(\mathbb{F}_p^{alg})$, 
we get that $\beta$ is a $\mathcal{L}_{ring}$-definable homeomorphism between $M_q(X_i^\eta)(\mathbb{F}_p^{alg})$ and $M_{q'}(Y_i^\eta)(\mathbb{F}_p^{alg})$ which restrict to a bijection between $M_q(X_i^\eta)(\mathbb{F}_{p^t})$ to $M_{q'}(Y_i^\eta)(\mathbb{F}_{p^t})$. 
Now it is enough to prove that \[|M_{q'}(Y_i^\eta)(\mathbb{F}_{p^t})|\geq p^{dt}-q^cp^{t(d-1/2)},\] for some constant $c$. 

By assumption, $a\in M_q(X_i^\eta\setminus X_\xi^\eta)(\mathbb{F}_{p^t})$. Break $M_q(X_i^\eta\setminus X_\xi^\eta)$ into irreducible (locally closed) components over $\mathbb{F}_{p^t}$ of the form $W':=W\setminus M_q(X_\xi^\eta)$ where $W$ is an irreducible subvariety of $M_q(X_i^\eta)$ over $\mathbb{F}_{p^t}$.  Then $a$ is in some $W'=W\setminus M_q(X_\xi^\eta)$. By Corollary~\ref{thm-equiDim}, $W$ contains an absolutely irreducible subvariety of algebraic dimension $d$ (as $X_{i,s}^{\eta}$ is a subvariety of $X_\xi^\eta$) and hence is of algebraic dimension $d$. Now $\beta(W)$ is an irreducible subvariety of $M_{q'}(Y_i^\eta)$ over $\mathbb{F}_{p^t}$. Recall that $Q$ is a $\tau_i$-difference polynomial which vanishes on $Y_i$. Let $\{P_1:=Q_{q'},\ldots,P_m\}$ be a set of defining (algebraic) polynomials for $\beta(W)$, where $Q_{q'}$ is the algebraic polynomial obtained by replacing $\tau_i$ in $Q$ by $\frob_{q'}$.

As $(\partial Q/\partial t_1)(\beta(a))\neq 0$ and by the definition of partial derivative for difference equations, we get $(\partial Q_{q'}/\partial t_1)(\beta(a))=(\partial Q/\partial t_1)(\beta(a))\neq 0$. Hence, the Jacobian $(\frac{\partial P_i}{\partial t_j}(\beta(a))_{i,j}$ has rank at least $1$. By Fact~\ref{lem-JacobianCriterion}, $\beta(a)$ is a smooth point of $\beta(W)$. Now $\beta(W)$ is an irreducible variety over $\mathbb{F}_{p^t}$ which contains a smooth point in $\mathbb{F}_{p^t}$, hence is absolutely irreducible.

Since $\beta(W)$ is an absolutely irreducible component of $M_{q'}(Y_i^\eta)$, by Lemma~\ref{lem-degreeBound} there is $c_0$ depending only on the complexity of $Y_i^\eta$ such that $\deg(\beta(W))\leq q^{c_0}$. This holds indeed since $q$ and $q'$ are bounded multiples of each other. Note that the complexity of $Y_i^\eta$ only depends on that of $X$ by Remark~\ref{rem: uniformstratification}. By Fact~\ref{lem-counting}, if $p^t>2(d+1)q^{2c_0}$, then 
$$||\beta(W)(\mathbb{F}_{p^t})|-p^{td}|\leq(q^{c_0}-1)(q^{c_0}-2)p^{td-\frac{t}{2}}+5q^{\frac{13c_0}{3}}p^{td-t}.$$

We may choose $c$ large enough (only depending on $c_0$) so that we have
\[|M_{q'}(Y_i^\eta)(\mathbb{F}_{p^t})|\geq |(\beta(W)(\mathbb{F}_{p^t})|\geq p^{td}-q^{c}p^{t(d-\frac{1}{2})}. \qedhere \]
\end{proof}

\section{Coarse dimension in pseudofinite difference fields}\label{sec:psdf}
In this section, we give an application of our main result to pseudofinite difference fields. In \cite{ZouDifference}, TZ studied the model theory of a family of ultraproducts of finite difference fields. It was shown that these structures are not model-theoretically tame, they have TP$_2$, the strict order property, and are not decidable. However, it was observed that the non-standard counting size of definable sets in these structures gives rise to an integer-valued dimension. It was conjectured in \cite[Conjecture 3.1]{ZouDifference} that this dimension should coincide with the transformal dimension for a quantifier-free definable set. In this section, we will prove this conjecture and extend it to existentially definable sets. Moreover, we give an example where the two dimensions do not coincide for $\forall\exists$-definable sets in characteristic $p>0$.

\begin{definition}
Let $DF(p,n,m)$ be the difference fields with $p^n$-many elements and the automorphism $\sigma(t):=t^{p^m}$, namely $DF(p,n,m):=(\mathbb{F}_{p^n},\frob_{p^m})$.
\end{definition}
\begin{remark}
Note that $DF(p,n,m)=DF(p,n,m')$ where $m'=m \mod n$, hence we may assume $n>m$ for any structure $DF(p,n,m)$. Moreover, all finite difference fields are of this form.
\end{remark}

\begin{definition}
Let $M=\prod_{i\to\mathcal{U}}M_i$ be an ultraproduct of some finite $\mathcal{L}$-structures $M_i$ for some non-principal ultrafilter $\mathcal{U}$. Let $\alpha=(\alpha_i)_{i\to\mathcal{U}}\in\mathbb{R}^{\mathcal{U}}$ with $\alpha>\mathbb{R}$. Let the \emph{coarse dimension with respect to $\alpha$} be the function $\pmb{\delta}_\alpha:\mathcal{D}(M)\to\mathbb{R}^{\geq 0}\cup\{\pm\infty\}$ on all definable sets $\mathcal{D}(M)$ in $M$, defined as \[\pmb{\delta}_\alpha(\phi(M)):=st.\frac{\log|\phi(M)|}{\log \alpha}:=\lim_{i\to \mathcal{U}}\frac{\log |\phi(M_i)|}{\log \alpha_i},\] where we set $\log 0=-\infty$.
\end{definition}

In the following, we will fix an  ultraproduct $K=(K,\sigma):=\prod_{i\to\mathcal{U}}DF(p_i,n_i,m_i)$ over some non-principal ultrafilter $\mathcal{U}$) with $\lim_{i\to\mathcal{U}}m_i/n_i=0$ and $\lim_{i\to\mathcal{U}}p_i^{m_i}=\infty$. Let $\pmb{\delta}:=\pmb{\delta}_{\alpha}$ where $\alpha=(p_i^{n_i})_{i\to\mathcal{U}}$. 
Note that this is a more general class than the pseudofinite difference fields considered in~\cite{ZouDifference}, where one imposes strict conditions on the rate of convergence of $\lim_{i\in \mathcal{U}}n_i/m_i=0$.

\begin{lemma} \label{lem-deltaCount}
    Suppose $X\subseteq \Aa^n$ is a difference variety defined over some difference subfield $F\leq K$ such that $\tdim_F(X)=d>0$. 
Then there is a difference subvariety $X_\xi$ defined over $F$ of transformal dimension $<d$, such that if $(X\setminus X_{\xi})(K)\neq \emptyset$, then $\pmb{\delta}(X(K))=d$.
\end{lemma}

\begin{remark}\label{remk-delta<=}
Note that we have the natural upper bound $\pmb{\delta}(X(K))\leq d$. Indeed, by Remark~\ref{rem-upperBound}, $|X(DF(p_i,n_i,m_i))|\leq p_i^{cm_i+dn_i}$ for some constant $c$ whenever $p_i^{m_i}$ is large enough. Hence, $\log{|X(DF(p_i,n_i,m_i))|}/\log{p_i^{n_i}}\leq(cm_i+dn_i)/n_i$ and $\pmb{\delta}{(X(K))}\leq \lim_{i\to\mathcal{U}}(cm_i+dn_i)/n_i=d$.
\end{remark}
\begin{proof}
  By the remark above, we only need to show that $\pmb{\delta}(X(K))\geq d$. 
  By Theorem~\ref{thm-finiteryCount}, there is a finitely generated difference domain $D'=\mathbb{Z}[a_1,\ldots,a_m]_{\sigma}\subseteq F$ and a difference subvariety $X_\xi$ of transformal dimension $<d$ defined over $D'$ (hence also over $F$), such that for some constants $c$ and $C$, for every Frobenius reduction $\eta:D'\to K_q$ with $q>C$, for all $t$ with $\eta(D')\subseteq \mathbb{F}_{p^t}$, if there is a point $x\in M_q(X^\eta\setminus X_\xi^\eta)(\mathbb{F}_{p^t})$, then \[|X^\eta(\mathbb{F}_{p^t},\frob_q)|\geq p^{dt}-q^cp^{t(d-1/2)}.\]

  Let $\bar{a}=(a_1,\ldots,a_m)$. Suppose $\bar{a}=(\bar{a}_i)_{i\to\mathcal{U}}$. Let $I\leq \mathbb{Z}[X_1,\ldots,X_m]_\sigma$ be the ideal of difference polynomials $P(X)$ with $P(\bar{a})=0$ in $K$. By Fact~\ref{RittRaud}, $I$ is finitely generated as a prime difference ideal. Fix a finite set of difference polynomials $(P_j)_{j\leq N}$ generating $I$, i.e., whenever $J\leq \mathbb{Z}[X_1,\ldots,X_m]_\sigma$ is a prime difference ideal with $P_j\in J$ for all $j\leq N$, then $I\subseteq J$. For ultrafilter-many $i$, $P_j(\bar{a}_i)=0$ holds in $D(p_i,n_i,m_i)$ for all $j\leq N$, hence there is a homomorphism $\eta_i:D'\to DF(p_i,n_i,m_i)$ by sending $\bar{a}$ to $\bar{a}_i$. This gives rise to a Frobenius reduction $\eta_i:D'\to K_{p_i^{m_i}}$ with $\eta_i(D')\subseteq DF(p_i,n_i,m_i)$. Let $x=(x_i)_{i\to\mathcal{U}}$ and $q_i:=p_i^{m_i}$. By assumption, $x_i\in M_{q_i}(X^{\eta_i}\setminus X_\xi^{\eta_i})(\mathbb{F}_{p_i^{n_i}})$ for ultrafilter-many $i$. Therefore, \[|X(DF(p_i,n_i,m_i))|\geq p_i^{dn_i}-p_i^{cm_i}p_i^{n_i(d-1/2)}>\frac{1}{2}p_i^{dn_i},\] for ultrafilter-many $i$. Hence, $\pmb{\delta}(X(K))\geq d$ as desired. 
\end{proof}

\begin{definition}
For a partial type $\pi$ (closed under conjunction) over $K$, we set $\pmb{\delta}(\pi):=\inf\pmb{\{\delta}(\phi(K))\mid \phi\in\pi\}$.
\end{definition}

\begin{theorem}\label{thm-main2}
 Let $A$ be a set of parameters in $K$. Suppose $p(x)=\qftp(a/A)$ where $a$ is a tuple in $K$ with $\mbox{trf.deg}(a/A)=d$. Then $\pmb{\delta}(p)=d$, moreover there is $\phi\in p$ such that $\pmb{\delta}(\phi)=\pmb{\delta}(p)$.

\end{theorem}

\begin{proof}
We may assume $A$ is a difference subfield of $K$.
By assumption, $\mbox{trf.deg}(a/A)=d$. Let $Y$ be defined by $I(a/A)$, then $\tdim_A(Y)=d$.

Let $\phi\in p(x)$, to prove $\pmb{\delta}(p)=d$, it is enough to show that $\pmb{\delta}(\phi(K)\cap Y(K))=d$, hence we may assume $\phi(x)$ implies $Y$. Note that $\phi(x)$ is a disjunction of conjunctions of equations and negations of equations of polynomials. Since $\pmb{\delta}$ of a finite union is the maximum, we may assume $\phi(x)$ defines a basic open difference variety $Z$, i.e. it is the non-vanishing set of some difference polynomial restricted to $Y$. Hence, $Z$ is in quantifier-free definable bijection $f$ with some difference variety $X$ over $A$. 
Since $Z\subseteq Y$ and $Z\in p(x)$, $d=\tdim_A(Z)=\tdim_A(X)$. It is enough to show that $\pmb{\delta}(X(K))=d$.

By Remark~\ref{remk-delta<=}, we only need to show $\pmb{\delta}(X(K))\geq d$.
Note that $\pmb{\delta}(X(K))\geq 0$ holds trivially if $d=0$. We may assume $d>0$. By Lemma~\ref{lem-deltaCount}, there is some $X_\xi$ defined over $A$ of transformal dimension $<d$, such that if $X(K)$ has a point outside $X_\xi(K)$, then $\pmb{\delta}(X(K))=d$. Thus, it is enough to show that $f(a)\not\in X_{\xi}(K)$ (as $f(a)\in X(K)$ by assumption). Note that $\trfdeg(f(a)/A)=d$, since $f$ is a definable bijection. We conclude $f(a)\not\in X_{\xi}(K)$, as $X_{\xi}$ has transformal dimension $<d$ over $A$.

\end{proof}

\begin{theorem}\label{lem: dimes equal for existential}
Let $A$ be a countable difference subfield of $K$. 
Given tuples $a,b$ in $K$, let $r(x,y):=\qftp(a,b/A)$ and let $d:=\trfdeg(a/A)$. Then $d=\pmb{\delta}(\exists y r(x,y))$, where $\exists y r(x,y):=\{\exists y\phi(x,y)\mid \phi(x,y)\in r\}$. 
\end{theorem}
\begin{proof}
Let $m\geq0$ such that $d+m=\trfdeg(a,b/A)$. Let $p(x)$ be the restriction of $r(x,y)$ to the $x$ variable. Since $r(x,y)$ is complete and quantifier-free, so is $p(x)$.

 By sub-additivity of $\pmb{\delta}$, see \cite[Section 5.1]{HrushovskiStable} and \cite[Proposition 2.31(5), Example 2.30(4)]{ArtemNotes}, we get \[\pmb{\delta}(r(x,y))\leq \pmb{\delta}(\exists y r(x,y))+\sup\{\pmb{\delta}(r(a',y)):a'\in K, \text{ exists }b'\in K, K\models r(a',b')\}.\] Choose $a'\in K$ which reaches the supreme on the right-hand side (it is possible by Theorem~\ref{thm-main2}). Hence \[\pmb{\delta}(r(x,y))\leq \pmb{\delta}(\exists y r(x,y))+\pmb{\delta}(r(a',y)).\]

  Since $r(x,y)=\qftp(a,b/A)=\qftp(a',b'/A)$ and $r(a',y)=\qftp(b'/A,a')$ for some $a,b,a',b'\in K$. Then $\trfdeg(a',b'/A)=d+m$ and $\trfdeg(a'/A)=d$ as $a'\models p(x)$. Therefore, $\trfdeg(b'/A,a')=m$.
  By Theorem \ref{thm-main2}, $\pmb{\delta}(r(a',y))=m$. Hence, again by Theorem \ref{thm-main2}, \[d+m=\pmb{\delta}(r(x,y))\leq \pmb{\delta}(\exists y r(x,y))+m,\] and thus $d\leq \pmb{\delta}(\exists y r(x,y))$. 

  For the other direction, note that $p(x)$ is a subcollection of $\exists y r(x,y)$, and so by Theorem \ref{thm-main2}, $d=\pmb{\delta}(p(x))\geq \pmb{\delta}(\exists y r(x,y))$.
\end{proof}

\begin{corollary}
 For any existential-definable set $\phi(x)$ over a countable subset $A$ of $K$, \[\pmb{\delta}(\phi(K))=\max\{\trfdeg(a/A):a\in\phi(K).\}\] In particular, the coarse dimension $\pmb{\delta}$ takes integer value for any existential-definable set in $K$.
\end{corollary}
\begin{remark}
It is possible to have a difference variety $X$ defined over $A$ with transformal dimension $d$, but $\max\{\mbox{trf.deg}(a/A):a\in X(K)\}<d$. 
\end{remark}
\begin{proof}

    Note that given any existential formula $\phi(x)=\exists y\psi(x,y)$ over $A$, for any $a\in \phi(K)$, there is some $b$ in $K$ such that $K\models \psi(a,b)$. Let $r(x,y):=\qftp(a,b/A)$ and then $\phi(x)$ is contained in the partial type $\exists y r(x,y)$. By definition, $\pmb{\delta}(\phi(K))\geq\pmb{\delta}(\exists y r(x,y))=\trfdeg(a/A)$. 
    Therefore, $\pmb{\delta}(\phi(K))\geq \max\{\mbox{trf.deg}(a/A):a\in\phi(K)\}$. 
    
    For the other direction, for any $a\in\phi(K)$, let $p(x)=\qftp(a/A)$, by Theorem \ref{thm-main2}, there is a quantifier-free $\psi_a(x)\in p(x)$, such that $\pmb{\delta}(\psi_a(K))=\trfdeg(a/A)$. By construction, 
    $\phi(K)\subseteq \bigcup_{a\in\phi(K)}\psi_a(K)$. So by compactness, since there are countably many quantifier-free formulas over $A$, $\phi(K)$ is covered by finitely many of them. Hence $\pmb{\delta}(\phi(K))\leq \max\{\pmb{\delta}(\psi_a(K))=\mbox{trf.deg}(a/A):a\in\phi(K)\}$, and we get the desired equality.
\end{proof}

\begin{remark}
    In the setting of Theorem~\ref{thm-main2}, one might wonder if $K=(K,\sigma)$ is \emph{transformally PAC}, just as $K$ is PAC (pseudo-algebraically closed) as a pure field.\footnote{We would like to thank Yuval Dor for raising this question and for engaging in helpful discussions around it.} One natural definition of transformally PAC would be the following: For a perfect inversive difference field $F$  and a difference field extension $F\subseteq E$, we say $E$ is a \emph{$\sigma$-regular extension of $F$}, if $F$ is relatively  $\sigma$-algebraically closed in $E$, namely for any $a\in E$, we have $a\in F$ whenever $\trfdeg(a/F)=0$. And $F$ is called transformally PAC if $F$ is existentially closed in any $\sigma$-regular extension.

    However, $K$ is not transformally PAC under this definition. Suppose the characteristic of $K$ is not $2$. Choose $c\in K$ not a square in $K$. Consider the polynomial $P(X,Y)$:= $X\sigma(X)-cY^2$. Then $K$ has no 
    solution of $P(X,Y)=0$. Since otherwise, if $(a,b)=(a_i,b_i)_{i\to\mathcal{U}}$ is a solution and $c=(c_i)_{i\to\mathcal{U}}$. Then, for ultrafilter-many $i$, $a_ia_i^{p_i^{m_i}}-c_ib_i^2=0$ and hence $c_i$ is a square, as $p_i^{m_i}+1$ is even, contradicting that $c$ is not a square. Let $x$ be transformally transcendental over $K$ and $y$ be such that $cy^2=x\sigma(x)$. Let $E=K(x,y)_\sigma$. By construction, $P(x,y)=0$, so $P(X,Y)$ has a solution in $E$. Let $y_0=y$ and $y_i$ be such that $y_i^2=\sigma^i(x)\sigma^{i+1}(x)/\sigma^i(c)$. Note that $E$ as a pure field is generated by $x,y_0,y_1,\ldots,y_i,\ldots$. It is clear that $E_i=K(x,y_0,\cdots,y_i)$ is purely transcendental over $E_{i-1}$ by construction. Thus $E=\bigcup_iE_i$ is purely transcendental over $K$.
    However, $K$ is relatively $\sigma$-algebraically closed in $E$. Indeed, $x,\sigma(x),\sigma^2(x),\ldots$ is a transcendence basis of $E/K$. Suppose there is $(K,\sigma)\subsetneq (L,\sigma)\subseteq (E,\sigma)$ with $(L,\sigma)$ $\sigma$-algebraic over $K$. Note also that $L/K$ is not algebraic. Then $x,\sigma(x),\sigma^2(x),\ldots$ are algebraically dependent over $L$. Hence, $E$ is a $\sigma$-algebraic extension of $L$, and is also $\sigma$-algebraic over $K$ (since $L/K$ is), a contradiction. 
\end{remark}

One might wonder how far Lemma~\ref{lem: dimes equal for existential} can be extended beyond existential formulas. 
In \cite{ZouDifference}, it was conjectured that Lemma~\ref{lem: dimes equal for existential} is true for all formulas in pseudofinite difference fields $K=\prod_{i\to\mathcal{U}} DF(p_i,n_i,m_i)$, as long as $n_i$ grows sufficiently faster than $m_i$. In the same paper, it was proved that under this condition, the coarse dimension $\pmb{\delta}$ of any definable set is integer-valued and is bounded above by the maximum of transformal transcendence degrees of elements in this definable set. However, this is not true when the characteristic of $K$ is positive. Here we include an example from the unpublished work-in-progress of Michael Benedikt and EH.

\begin{example}
    Let $K=\prod_{i\to\mathcal{U}} DF(p,n_i,m_i)$ with $m_i\vert n_i$ and $\lim_{i\to\mathcal{U}} p^{m_i}/n_i=0$. Let $F=\prod_{i\to\mathcal{U}}\mathbb{F}_{p^{m_i}}$ be the fixed field of $K$. Let $y\in K$  be such that $y\not\in \acl(F)$. Let $X:=\{1/(x-y):x\in F\}$. Consider $K$ as an $\mathbb{F}_p$-vector space. Note that $X$ is linearly independent over $\mathbb{F}_p$, since otherwise, $y$ will satisfy some non-trivial polynomial over $X$. We claim that the $\mathbb{F}_p$-span $\langle X\rangle$ of $X$ is definable. Indeed, by \cite[Lemma 3.9]{MS08}, $K$ has a non-degenerate symmetric bilinear form $b(-,-):K\times K\to\mathbb{F}_p$ which is $\mathcal{L}_{ring}$-definable. Now it is easy to verify that $\langle X\rangle$ is defined by $\psi(x):=\forall u\forall z(z\in X\land b(u,z)=0\to b(u,x)=0)$. Suppose $Y:=\langle X\rangle=\prod_{i\to\mathcal{U}} Y_i$. Then clearly, $|Y_i|\leq p^{p^{m_i}}$. Hence, $\pmb{\delta}(Y)=0$ by the assumption that $\lim_{i\to\mathcal{U}} p^{m_i}/n_i=0$. We claim that $|Y|\geq |X|^n$ for any $n$. Indeed, given $n$, define a function $f:X^n\to Y$ by $(a_1,\ldots,a_n)\mapsto \sum_{i\leq n} a_i$. Then for any tuple $(a_1,\ldots,a_n)$ satisfying $a_i\neq a_j$ for $i\neq j$, we have $f(a_1,\ldots,a_n)=f(b_1,\ldots,b_n)$ if and only if $(b_1,\ldots,b_n)$ is a permutation of $(a_1,\ldots,a_n)$. Hence, $|Y|\geq |f(X^n)|\geq \frac{1}{n!}|X|^n\geq |X|^{n-1}$. But the last claim implies that $\max\{\trfdeg(a):a\in Y\}>0$. Since otherwise, by compactness, $Y\subseteq Z$ where $Z$ is defined by some non-trivial difference polynomial equation $P(x)=0$. And by Remark~\ref{rem-upperBound}, $|Z(DF(p,n_i,m_i))|\leq p^{cm_i}$ for some constant $c$. Hence, $|Y|\leq |Z(K)|\leq |X|^c$, a contradiction.
\end{example}

\bibliographystyle{alpha}
\bibliography{ref}
\end{document}